\RequirePackage{fix-cm}
\documentclass[smallextended,numbook]{svjour3}       
\smartqed  
\usepackage{graphicx}%
\usepackage{multirow}%
\usepackage{amsmath,amssymb,amsfonts}%
\usepackage{mathrsfs}%
\usepackage[title]{appendix}%
\usepackage{xcolor}%
\usepackage{textcomp}%
\usepackage[title]{appendix}
\usepackage{manyfoot}%
\usepackage{booktabs}%
\usepackage{algorithm}%
\usepackage{algorithmicx}%
\usepackage{algpseudocode}%
\usepackage{listings}%
\usepackage{comment}
\usepackage{lipsum}
\usepackage{amsfonts}
\usepackage{bm}   
\usepackage{amsmath}
\usepackage{graphicx}
\usepackage{dsfont}
\usepackage{epstopdf}
\usepackage{comment}
\usepackage{booktabs}
\usepackage{enumerate}
\usepackage{multirow}
\usepackage{soul}
\usepackage{caption}
\usepackage{subcaption}
\usepackage{scalerel,stackengine}
\usepackage{tikz}
\usepackage{mathtools}
\usepackage{amsopn}
\usepackage{algorithm}
\usepackage{cite}
\usepackage{algpseudocode}
\usepackage{cleveref}
\usepackage{url}

\newtheorem{prop}{Proposition}
\newtheorem{rmk}{Remark}[section]

\def\RR{\mathbb{R}}

\DeclareMathOperator*{\argmax}{arg\,max}

\DeclareMathOperator*{\argmin}{arg\,min}

\newcommand{\reallywidehat}[1]{%
\savestack{\tmpbox}{\stretchto{%
\scaleto{%
\scalerel*[\widthof{\ensuremath{#1}}]                         {\kern-.6pt\bigwedge\kern-.6pt}%
{\rule[-\textheight/2]{1ex}{\textheight}}
}{\textheight}%
}{0.5ex}}%
\ensurestackMath{\stackon[1pt]{#1}{\tmpbox}}%
}

\DeclarePairedDelimiter\floor{\lfloor}{\rfloor}

\newcommand{\R}{\mathbb R}

\begin{document}

\title{{A multilinear HJB-POD method for \\  the optimal control of PDEs

\vspace{3mm}

 \textit{\small In loving memory of Maurizio Falcone}}
}

\titlerunning{A multilinear HJB-POD method for the optimal control of PDEs}        

\author{Gerhard Kirsten         \and
        Luca Saluzzi 
}


\institute{G. Kirsten \at
              Dipartimento di Matematica,Universit\`a di Bologna,Piazza di Porta S. Donato, 5, Bologna, I-40127, Italy. \\
              \email{gpkirsten@gmail.com}           
           \and
           L. Saluzzi \at
              Department of Mathematics, Imperial College London, South Kensington Campus, London, SW7 2AZ, UK. \\
              \email{lsaluzzi@ic.ac.uk} 
}

\date{Received: date / Accepted: date}

\maketitle

\begin{abstract}
Optimal control problems driven by evolutionary partial differential equations arise in many industrial applications and their numerical solution is known to be a challenging problem. One approach to obtain an optimal feedback control is via the Dynamic Programming principle. Nevertheless, despite many theoretical results, this method has been applied only to very special cases since it suffers from the \textit{curse of dimensionality}. Our goal is to mitigate this crucial obstruction developing a new version of dynamic programming algorithms based on a tree structure and exploiting the compact representation of the dynamical systems based on tensors notations via a model reduction approach. Here, we want to show how this algorithm can be constructed for general nonlinear control problems and to illustrate its performances on a number of challenging numerical tests. Our numerical results indicate a large decrease in memory requirements, as well as computational time, for the proposed problems. Moreover, we prove the convergence of the algorithm and give some hints on its implementation.
\end{abstract}

\section{Introduction}\label{sec1}

Feedback control is a fundamental concept in engineering and applied mathematics, where the goal is to design a system that can regulate a process to achieve a desired behavior. One of the most powerful tools in feedback control is the Hamilton Jacobi Bellman (HJB) equation, which provides a framework for optimal control of dynamical systems. The HJB equation is a partial differential equation that arises from the calculus of variations and has a wide range of applications, including $e.g.$ robotics, aerospace and finance.
The main disadvantage of this approach comes in the form of the so-called \textit{curse of dimensionality}; the phenomenon for which the complexity of a problem increases exponentially as the number of variables or dimensions involved in the problem grows. In real applications, the dynamical system may be described by a large number of state variables either because the continuous problem is in high-dimension or because the dynamical system is obtained via a discretization in space of a Partial Differential Equation (PDE). In the context of linear dynamics and quadratic cost functional, the HJB is equivalent to the Differential Riccati Equation for the finite horizon control problem and to the Algebraic Riccati Equation in the infinite horizon case. This setting has been widely researched, leading to several promising high-dimensional solvers \cite{kirsten2019,BBKS_2020}.

For general nonlinear problems, such a reformulation does not exist and the HJB equation must be tackled directly. 
In the recent years several efforts have been employed in the mitigation of the curse of dimensionality arising in optimal control, among those we mention sparse grids \cite{GK16}, max-plus algebra \cite{McEneaney_2007,Akian_Gaubert_Lakhoua_2008,maxplusdarbon,akian2023adaptive}, artificial neural networks \cite{Han_Jentzen_E_2018,Darbon_Langlois_Meng_2020,Kunisch_Walter_2021,sympocnet,Zhou_2021,Onken2021,ruthotto2020machine}, the application of tensor formats \cite{dolgov2022data,oster2022approximating,richter2021solving} and radial basis functions \cite{alla2021hjb}.

In this paper we aim to mitigate the curse of dimensionality via a graph-based optimization algorithm, the Tree Structure Algorithm (TSA) for the resolution of the finite horizon HJB problem \cite{AFS19}. The TSA leads to the construction of a tree in the direction of all the possible controlled trajectories. Due to its flexible structure, this technique has been already applied in different settings, $e.g.$ state constraints problems \cite{afs20} and high-dimensional semidiscrete PDEs \cite{Alla_Saluzzi_2020} and its convergence is ensured by rigorous error estimates \cite{saluzzi2022error}. Furthermore, a \textit{geometrical pruning} based on the distance of the nodes has been introduced to avoid the exponential growth of the tree and obtain a quadratic growth rate in the context of LQR problems \cite{saluzzi2022error}. Unfortunately, for general nonlinear problems this criterion may be not effective since the tree nodes may spread out faster, leading to a further curse of dimensionality. This is the first shortcoming of the TSA that we will aim to address in this paper. We will investigate:
\begin{itemize}
    \item a bilinear setting where the application of the geometrical pruning also yields a good reduction in the cardinality of the tree,
    \item an optimal control problem based on monotone controls with an efficient tree-based data structure,
    \item a \textit{statistical pruning} rule based on the iterative knowledge of the value function on the tree nodes.
\end{itemize} 

A second shortcoming of the TSA that we address in this paper is related to the computational cost of evaluating and constructing full-dimensional tree nodes. Given the exponential growth in the cardinality of the tree, that can merely be mitigated by pruning techniques, a massive computational effort may be required to construct and evaluate the discrete problem on the tree nodes, as the dimension of the discrete problem is increased. 

A first attempt to address this issue was proposed in \cite{Alla_Saluzzi_2020} where the authors applied a combination of the Proper Orthogonal Decomposition (POD) \cite{volkwein2011model} for the linear part of the problem and Discrete Empirical Interpolation method (DEIM) \cite{chaturantabut2010nonlinear} for the nonlinear terms, to reduce the complexity of the problem.
The coupling of the POD technique and the HJB equation dates back to the pioneering paper by Kunisch and co-authors \cite{KVX2004}, which was then further developed in a series of works \cite{KX2005,kunisch2010optimal,HV2005}. Nevertheless, in this setting, the POD-DEIM algorithm itself has some computational drawbacks. More precisely, the dimension of the vectors that need to be stored for constructing the tree and the POD and DEIM spaces increase exponentially to the order of the dimension of the underlying dynamical system. This may lead to a large bottleneck in the computation of the reduced spaces for larger problems.

Instead, in this paper, we take advantage of the composite structure of a subset of semilinear PDEs, where the underlying PDEs can be written and discretized in Array form, leading to discrete semilinear Matrix and Tensor equations \cite{Simoncini2017,Autilia2019matri,palitta2016,kirsten.22}. Given this particular structure, we aim to show how the tree can be constructed in low-dimension by applying a Higher-Order POD-DEIM (HO-POD-DEIM) model order reduction \cite{kirsten.22} to the discrete problem and then solving the HJB equation on the low-dimensional tree. Our computational results on several benchmark problems indicate that the new algorithm leads to memory requirements that increase linearly in the dimension of the underlying dynamical system, instead of exponentially. Furthermore, the convergence of the proposed technique is established by the derivation of rigorous error estimates for the reduced discrete dynamical system and for the discrete value function computed on the reduced tree. These theoretical results extend the error bounds obtained in \cite{sorensen2016} to semi-implicit schemes as well as the proposed HO-POD-DEIM technique.

Our construction focuses on general high-dimensional semidiscretized PDEs. A
simplified matrix-oriented version of our framework is experimentally explored for the Navier-Stokes (NS) equation in the companion manuscript \cite{KSF2023}, where the application to systems of differential equations
is also discussed. This is a classical example known to be computationally expensive, where the application of MOR techniques helps in the computation of the solution (see $e.g.$ \cite{QR2007,SR2018,pichi2022driving}). Here we consider complex problems in high dimension, showing the promising numerical results on benchmark problems. Furthermore, we deepen the analysis of all the ingredients of this new
methodology, including pruning techniques, error estimates and important implementational nuances. We believe that the numerical simulations presented
in the last section illustrate that DP is now also feasible for more complex, higher-dimensional problems from a computational point of view, and we hope that this brings it closer to the application of challenging industrial problems.

 The paper is organized as follows. In the second section we introduce the optimal control framework and the Tree Structure Algorithm. Section 3 is devoted to the Model Order Reduction setting and its coupling with the TSA, whereas in Section 4 we present some hints for an efficient implementation of the proposed algorithm. In Section 5 we examine different pruning criteria for the TSA showing some results in the reduction of the cardinality of the tree, and Section 6 presents an error bound for the approximation of the value function via the reduced order model algorithm. Finally, in the last section we present some numerical experiments to show the effectiveness of the proposed method.

\section{The optimal control problem}
Let us consider the classical {\it finite horizon optimal control problem} that we use as a model problem. The system is  driven by
\begin{equation}\label{eq}
\left\{ \begin{array}{l}
\dot{y}(s)=f(y(s),u(s),s), \;\; s\in(t,T],\\
y(t)=x\in\R^N.
\end{array} \right.
\end{equation}
Here, $y:[t,T]\rightarrow\R^N$ is the solution, $u:[t,T]\rightarrow\R^m$ is the control, $f:\R^N\times\R^m\times[t,T]\rightarrow\R^N$ is the dynamics and 
\[\mathcal{U}=\{u:[t,T]\rightarrow U, \mbox{measurable} \}
\]
is the set of admissible controls where $U\subset \R^m$ is a compact set. 
We define the cost functional for the finite horizon optimal control problem as
\begin{equation}\label{cost}
 J_{x,t}(u):=\int_t^T L(y(s,u),u(s),s)\, ds+g(y(T)),
\end{equation}
where $L:\R^N\times\R^m\times [t,T]\rightarrow\R$ is the running cost and $g:\R^N\rightarrow\R$ is the final cost. In the present work we will assume that the functions $f,L$ and $g$ are bounded:
 \begin{align}
 \begin{aligned}\label{Mf}
|f(x,u,s)| \le M_f,&\quad |L(x,u,s)| \le M_L,\quad |g(x)| \le M_g, \cr
&\forall\, x \in \mathbb{R}^N, u \in U \subset \mathbb{R}^m, s \in [t,T], 
\end{aligned}
\end{align}
the functions $f$ and $L$ are Lipschitz-continuous with respect to the first variable
\begin{align}
\begin{aligned}\label{Lf}
&|f(x,u,s)-f(y,u,s)| \le L_f |x-y|, \quad |L(x,u,s)-L(y,u,s)| \le L_L |x-y|,\cr
&\qquad\qquad\qquad\qquad\forall \, x,y \in \mathbb{R}^N, u \in U \subset \mathbb{R}^m, s \in [t,T], 
\end{aligned}
\end{align}
%
and finally the cost $g$ is also Lipschitz-continuous:
\begin{equation}
|g(x)-g(y)| \le L_g |x-y|, \quad \forall x,y \in \mathbb{R}^N.
\label{Lg}
\end{equation}

Note that these assumptions guarantee uniqueness for the trajectory $y(t)$ by the Carath\'eodory theorem  (we refer to e.g. \cite{BC08} for a precise statement).

The aim is to construct a state-feedback control law $u(t)=\Phi(y(t),t),$ in terms of the state equation $y(t),$ where $\Phi$ is the feedback map. The optimality conditions are derived via the well-known Dynamic Programming Principle (DPP) introduced by R. Bellman. We first introduce the value function for an initial datum $(x,t)\in\R^N\times [t,T]$:
\begin{equation}
v(x,t):=\inf\limits_{u\in\mathcal{U}} J_{x,t}(u)
\label{value_fun}
\end{equation}
which can be represented via the DPP, i.e. for every $\tau\in [t,T]$:
\begin{equation}\label{dpp}
v(x,t)=\inf_{u\in\mathcal{U}}\left\{\int_t^\tau L(y(s),u(s),s) ds+ v(y(\tau),\tau)\right\}.
\end{equation}
Due to \eqref{dpp} the HJB can be derived for every $x\in\R^N$, $s\in [t,T)$: 
\begin{equation}\label{HJB}
\left\{
\begin{array}{ll} 
&-\dfrac{\partial v}{\partial s}(x,s) + \max\limits_{u\in U }\left\{-L(x, u,s)- \nabla v(x,s) \cdot f(x,u,s)\right\} = 0, \\
&v(x,T) = g(x).
\end{array}
\right.
\end{equation}
Once the value function is known, by e.g. solving \eqref{HJB}, then the optimal feedback control can be obtained as:
\begin{equation}\label{feedback}
u^*(t):=  \argmax_{u\in U }\left\{-L(x,u,t)- \nabla v(x,t) \cdot f(x,u,t)\right\}. 
\end{equation}

\subsection{Dynamic programming on a tree structure}\label{sec2.1}
We briefly sketch the essential features of the dynamic programming approach on a tree based on the discrete approximation of the dynamical system. More details on the tree structure algorithm can be found in \cite{AFS19} where the algorithm and several tests have been presented.\\
It is hard to find analytical solutions of the HJB equation \eqref{HJB} due to the nonlinearity and classical approximation methods, e.g. finite difference or semi-Lagrangian schemes, need a space discretization that is impossible to manage in high-dimension (see the book \cite{FF13} for a comprehensive analysis of approximation schemes for Hamilton-Jacobi equations). This has motivated different approaches to mitigate the "curse of dimensionality".\\

We consider the discretized problem with a time step $\Delta t: = [(T-t)/N_t]$ where $N_t$ is the number of temporal time steps
\begin{equation}\label{SL}
\left\{\begin{array}{ll}
V^{n}(x)=\min\limits_{u\in U}[\Delta t\,L(x, u, t_n)+V^{n+1}(x+\Delta t f(x, u, t_n))], \quad
 n= N_t-1,\dots, 0,\\
V^{N_t}(x)=g(x),\hspace{7.5cm} x \in \R^N,
\end{array}\right.
\end{equation}
where $t_n=t+n \Delta t,\, t_{N_t} = T$, and $V^n(x):=V(x, t_n).$
The classical approach computes the solution through the application of an interpolation operator to obtain the term $V^{n+1}(x+\Delta t f(x, u, t_n))$ based on the values sitting on the grid nodes. This direction will be abandoned to build a tree structure and computing \eqref{SL} only on a tree structure.
Starting from the initial condition $x$, we consider all the nodes obtained following the  discrete dynamics, e.g. for the explicit Euler scheme with different discrete controls $u_j$. This gives in one step the points
\begin{equation}\label{def:ddyn}
\zeta_j^1 = x+ \Delta t \, f(x,u_j,t_0),\qquad j=1,\ldots,M.
\end{equation}
We assume that the control set $U$ is a bounded subset in $\mathbb{R}^m$ and
we discretize the control domain $U$ with constant step-size $\Delta u$ obtaining a discrete control set with a finite number of points $U^{\Delta u}=\{u_1,...,u_M \}$ that in the sequel we continue to denote by $U$ (with a slight abuse of notation).\\
Therefore, from every point $x$ we can reach $M$ points by \eqref{def:ddyn}. Identifying the root of the tree with $\mathcal{T}^0=\{x\}$ we obtain the first level of the tree  $\mathcal{T}^1 =\{\zeta_1^1,\ldots, \zeta^1_M\}$. We can proceed in  this way so that all the nodes at the $n-$th {\em time level}, will be given by 
$$\mathcal{T}^n = \{ \zeta^{n-1}_i + \Delta t f(\zeta^{n-1}_i, u_j,t_{n-1}) \}_{j=1}^{M}\quad i = 1,\ldots, M^{n-1}$$ 
and all the nodes belonging to the tree can be shortly defined as
 $$\mathcal{T}:= \{ \zeta_j^n  \}_{j=1}^{M^n},\quad n=0,\ldots N_t,$$ 
where the nodes $\zeta^n_i$ are the result of the dynamics at time $t_n$ with the controls $\{u_{j_k}\}_{k=0}^{n-1}$:
$$\zeta_{i_n}^n = \zeta_{i_{n-1}}^{n-1} + \Delta t f(\zeta_{i_{n-1}}^{n-1}, u_{j_{n-1}},t_{n-1})= x+ \Delta t \sum_{k=0}^{n-1} f(\zeta^k_{i_k}, u_{j_k},t_k), $$
with $\zeta^0 = x$, $i_k = \floor*{\dfrac{i_{k+1}}{M}}$ and $j_k\equiv i_{k+1} \mbox{mod } M$, where $\zeta_i^k \in \R^N, i=1,\ldots, M^k$ and $\floor*{\cdot}$ is the ceiling function.

Despite the fact that the tree structure allows the resolution of high dimensional problems, the construction may be expensive since $|\mathcal{T}| = O(M^{N_t})$,
where $N_t$ the number of time steps and $M$ is the number of controls. Whenever $M$ or $N_t$ are too large, the construction turns out to be infeasible due to the memory allocation. In Section \ref{sec:prun} we will introduce two pruning criteria and theoretical results on the reduction of the cardinality, showing their efficiency in avoiding the allocation memory problem.

Once the tree structure $\mathcal{T}$ has been constructed, we compute the numerical value function $V(x,t)$ on the tree nodes as 
\begin{equation}\label{num:vf}
V(x,t_n)=V^n(x), \quad \forall x \in \mathcal{T}^n, 
\end{equation}
where $t_n=t+ n \Delta t$. It is now straightforward to evaluate the value function. Since the TSA defines a grid $\mathcal{T}^n=\{\zeta^n_j\}_{j=1}^{M^n}$ for $n=0,\ldots, N_t$,
we can approximate \eqref{HJB} as follows: 
\begin{equation}
\begin{cases}
V^{n}(\zeta^n_i)= \min\limits_{u\in U} \{ V^{n+1}(\zeta^n_i+\Delta t f(\zeta^n_i,u,t_n)) +\Delta t \, L(\zeta^n_i,u,t_n) \}, \\
\qquad\qquad \qquad\qquad \qquad\qquad \qquad\qquad \qquad\qquad  \zeta^n_i \in \mathcal{T}^n\,, n = N_t-1,\ldots, 0, \\
V^{N_t}(\zeta^{N_t}_i)= g(\zeta_i^{N_t}), \qquad\qquad \qquad\qquad \qquad\qquad   \zeta_i^{N_t} \in \mathcal{T}^{N_t},
\end{cases}
\label{HJBt2}
\end{equation}

where the minimization is computed by comparison on the discretized set of controls $U$. 

\section{Reduced order models on a tree structure} 
Despite the fact that the tree structure algorithm avoids the construction of a grid in high dimensions, the resulting memory requirements
can still be overwhelming. A first step towards relieving this computational demand via model order reduction was presented in \cite{Alla_Saluzzi_2020}.
More precisely, the POD-DEIM algorithm from \cite{chaturantabut2010nonlinear} is used to reduce the dimension of the discrete dynamical 
system, so that the the tree construction is performed in low dimension. Nevertheless, the POD-DEIM algorithm itself has some computational drawbacks.
Firstly, if the discrete dynamical system from the finite difference semi-discretization of a PDE in dimension $d$, the memory requirements in both the offline and online phases
of POD-DEIM are of ${\cal O}(N)$, where $N = \prod_{i = 1}^{d} n_i$, where $n_i$ is the number of discretization nodes in the $i$th spatial direction. A similar increase in
memory requirements is experienced for other discretization techniques.

Instead, for high-dimensional semi-discrete PDEs, we couple the tree structure algorithm with the multilinear POD-DEIM algorithm presented in \cite{Kirsten.Simoncini.arxiv2020} for the 2D case and in \cite{kirsten.22}
for higher dimensions. This will decrease the memory requirements to ${\cal O}(\widetilde{N})$, with $\widetilde{N} = \sum_{i = 1}^d n_i$.

To this end, we will first discuss some basic tensor notation required for the new algorithm, after which we will review the standard HJB-POD algorithm, before introducing the new multilinear one.
\subsection{Notation and tensor basics}
The third mode of a third-order tensor ${\pmb {\cal T}} \in \RR^{n_1 \times n_2 \times n_3}$, is given by (see e.g., \cite{kolda2009})
$$
{\pmb {\cal T}}_{(3)} = \begin{pmatrix} {\bm T}_1, {\bm T}_2, \cdots, {\bm T}_{n_2} \end{pmatrix},$$
where
${\bm T}_i \in \RR^{n_3 \times n_1}, i = 1,2,\ldots,n_2$
 is referred to as a lateral slice, and ${\pmb {\cal T}}_{(3)}$ is a matrix in $\RR^{n_3 \times n_1n_2}$. Multiplication between a tensor and a matrix, is done via the $m-$mode product, which, for a tensor ${\pmb {\cal T}} \in \RR^{n_1 \times n_2 \times n_3}$ and a matrix ${\bm M} \in \RR^{n \times n_m}$, we express as
$$
{\pmb{\cal Q}} = {\pmb {\cal T}} \times_m {\bm M} \quad \iff \quad {\pmb{\cal Q}}_{(m)} = {\bm M}{\pmb {\cal T}}_{(m)},
$$
in the $m$-th mode.
The Kronecker product of two matrices ${\bm M} \in \RR^{m_1 \times m_2}$ and ${\bm N} \in \RR^{n_1 \times n_2}$ is defined as
$$
{\bm M} \otimes {\bm N} = \begin{pmatrix}
M_{1,1}{\bm N} & \cdots &  M_{1,m_2}{\bm N}\\
\vdots & \ddots & \vdots\\
M_{m_1,1}{\bm N} & \cdots &  M_{m_2,m_2}{\bm N}
\end{pmatrix} \in \RR^{m_1n_1 \times m_2n_2},
$$
and the vec$(\cdot)$ operator stacks the columns of a matrix one after the other to form a long vector. For a third order tensor, the vectorization is applied via the first mode unfolding. Furthermore,
\begin{equation}
\label{kronproperty}
({\bm M} \otimes {\bm N})\mbox{vec}({\bm X}) = \mbox{vec}({\bm N}{\bm X}{\bm M}^{\top}).
\end{equation}  
{As a result, if ${\pmb{\cal X}} \in \RR^{n_1 \times n_2 \times n_3}$, and ${\bm X} = {\pmb{\cal X}}_{(3)}^{\top}$, then
\begin{equation}
\label{kronproperty2}
({\bm L} \otimes {\bm M} \otimes {\bm N})\mbox{vec}\left({\pmb{\cal X}}\right) = \mbox{vec}\left(({\bm M} \otimes {\bm N}){\bm X}{\bm L}^{\top}\right).
\end{equation}}
More important properties include (see, e.g., \cite{golub13}): (i) $({\bm M} \otimes {\bm N})^{\top} = {\bm M}^{\top} \otimes {\bm N}^{\top}$; (ii) $ ({\bm M}_1 \otimes {\bm N}_1)({\bm M}_2 \otimes {\bm N}_2) = ({\bm M}_1{\bm M}_2 \otimes {\bm N}_1{\bm N}_2)$; and (iii) $\| {\bm M} \otimes {\bm N} \|_2 = \|{\bm M}\|_2\|{\bm N}\|_2$.

\subsection{POD-DEIM reduced dynamics}
 Consider the nonlinear dynamical system \eqref{eq}. In what follows we will assume without loss of generality that the linear and nonlinear terms on the right-hand side can be explicitly separated to yield a semilinear system of the form

\begin{equation}\label{eq2}
\left\{ \begin{array}{l}
\dot{\bm y}(s)= f({\bm y}(s),u(s),s):= {\bm L}{\bm y}(s) + {\bm  f}({\bm y}(s),u(s),s), \;\; s\in(t,T],\\
y(t)=x\in\R^N,
\end{array} \right.
\end{equation}
where ${\bm L} \in \RR^{N \times N}$ and ${\bm  f}:\R^N\times\R^c\times[t,T]\rightarrow\R^N$ is a continuous function in all arguments and locally Lipschitz-type with respect to the first variable.

In \cite{Alla_Saluzzi_2020} the authors reduce the number of variables involved in the system \eqref{eq2} by means of POD-DEIM. More precisely, consider a set of solution {\em snapshots} $y_i = y(s_i)$ collected at $n_s$ time instances $s_i$ in the timespan $(t, T]$, and consider the {\em snapshot matrix}
$$
{\bm S} = [y_1,y_2,\ldots, y_{n_s}] \in \RR^{N \times n_s}, \qquad {\cal S} = \mbox{Range}({\bm S}). 
$$
A POD basis of dimension $k \le n_s$ is obtained by orthogonal reduction of the matrix ${\bm S}$. That is, given the Singular Value Decomposition (SVD)
$$
{\bm S} = {\bm V}{\bm \Sigma}{\bm W}^{\top}, \qquad {\bm V}, {\bm W} \in \RR^{N \times n_s}, {\bm \Sigma} \in \RR^{n_s \times n_s}, 
$$
the POD basis is given by $\{{\bm v}_1,\ldots,{\bm v}_k\}$, where ${\bm V}_k = [{\bm v}_1,\ldots,{\bm v}_k] \in \RR^{N \times k}$ is the matrix of truncated left singular vectors related to the $k$ largest singular values contained on the diagonal of ${\bm \Sigma}$.

Given the matrix ${\bm V}_k$, the state vector ${\bm y}(s)$ can be approximated as ${\bm y}(s) \approx {\bm V}_k\widehat{\bm y}(s)$, for all $s \in (t, T]$, where $\widehat{\bm y}(s) \in \RR^k$ solves the reduced dynamical system
\begin{equation}\label{eq2reduced}
\left\{ \begin{array}{l}
\dot{\widehat y}(s)= {\bm V}^\top_k f({\bm V}_k\widehat{\bm y}(s),u(s),s):= {\bm V}_k^\top {\bm L} {\bm V}_k\widehat{\bm y}(s) +  {\bm V}_k^\top{\bm  f}({\bm V}_k\widehat{\bm y}(s),u(s),s), \\
\widehat{y}(t)={\bm V}_k^{\top} x.
\end{array} \right.
\end{equation}
To ensure that the reduced model can be simulated with a computational cost independent of $N$, we need to avoid lifting the nonlinear term before projection onto the low-dimensional space. Consequently, the Discrete Empirical Interpolation Method (DEIM) from \cite{chaturantabut2010nonlinear} is used to interpolate the nonlinear function.

To this end we consider an approximation of the form
$$
{\bm  f}({\bm V}_k\widehat{\bm y}(s),u(s),s) \approx {\bm \varPhi}_p\widehat{\bm  f}({\bm V}_k\widehat{\bm y}(s),u(s),s), \qquad \widehat{\bm  f}({\bm V}_k\widehat{\bm y}(s),u(s),s) \in \RR^p,
$$
where ${\bm \varPhi}_p = [{\bm \varphi}_1, \ldots, {\bm \varphi}_p] \in \RR^{N \times p}$, with $p \ll N$, and $\{{\bm \varphi}_1, \ldots, {\bm \varphi}_p\}$ is a POD basis of dimension $p$ obtained from the set of snapshots $\{{\bm  f}({y}_i,u(s_i),s_i)\}_{i = 1}^{n_s}$. The overdetermined system is solved by interpolation, ensuring that the left and right side of the equation is equal at $p$ selected points. That is, given the matrix ${\bm P} = [{\bm e}_{\rho_1},\ldots,{\bm e}_{\rho_p}] \in \RR^p$ containing a subset of columns of the identity matrix, we ensure that ${\bm P}^\top{\bm  f}({\bm V}_k\widehat{\bm y}(s),u(s),s) = {\bm P}^\top{\bm \varPhi}_p\widehat{\bm  f}({\bm V}_k\widehat{\bm y}(s),u(s),s)$, so that 
$$
{\bm  f}({\bm V}_k\widehat{\bm y}(s),u(s),s) \approx \widetilde{\bm  f}(\widehat{\bm y}(s),u(s),s) := {\bm \varPhi}_p({\bm P}^\top{\bm \varPhi}_p)^{-1}{\bm P}^\top{\bm  f}({\bm V}_k\widehat{\bm y}(s),u(s),s).
$$
Throughout this paper we deal with nonlinear functions that are evaluated element-wise, so that
$$
\widetilde{\bm  f}(\widehat{\bm y}(s),u(s),s) = {\bm \varPhi}_p({\bm P}^\top{\bm \varPhi}_p)^{-1}{\bm  f}({\bm P}^\top{\bm V}_k\widehat{\bm y}(s),u(s),s),
$$
and the nonlinear term is only evaluated at $p$ entries.

\subsection{A Multilinear HJB-POD-DEIM algorithm on a tree structure}

In this section we illustrate how, under certain hypotheses, the discrete system \eqref{eq2} can be expressed, integrated and reduced in terms of multilinear arrays; see e.g., \cite{Simoncini2017,Autilia2019matri, palitta2016, kirsten.22}. We focus specifically on the case where the discrete system \eqref{eq2} stems from the space discretization of a semilinear PDE of the form
\begin{equation}\label{eq2pde}
\left\{ \begin{array}{l}
\small
\partial_s y(s,x) = {\cal L}\left(y(s,x)\right) + {\bm  f}(\nabla{{y}(s,x)},{y}(s,x),u(s),s),\\
\qquad \, \,  \, s\in(t,T], x \in \Omega, \\
y(t,x)=\tilde{y}(x), \;\;x \in \Omega,
\end{array} \right.
\end{equation}

with ${\cal L}$ a linear differential operator, ${\bm  f}$ a generic nonlinear operator and $\Omega \subset \RR^d$, for $d = 2,3$.
\subsubsection{Discretization in terms of multilinear arrays}

Consider the operator $\mathcal{L}$ to be a second order differential operator with separable coefficients. Then, the physical domain can be mapped to a hypercubic domain ${\bm  \Omega} = [a_1, b_1] \times \cdots \times [a_d, b_d]$, if the operator is discretized via a tensor basis. Examples of such discretizations include, but are not limited to, spectral methods, and finite differences on parallelepipedal domains. See, e.g., \cite{Simoncini2017, palitta2016, kirsten.22} for more information regarding the assumptions on the operators, domains and discretization techniques. Here we consider  $\mathcal{L}$ as a $d-$dimensional Laplace operator for illustration purposes, but more general operators can also be treated. Under these conditions, it holds (from \eqref{eq2}) that
$$
{\bm  L} = \sum_{m = 1}^{d} {\bm  I}_{n_d} \otimes \cdots \otimes \overset{m}{\bm  A}_{m} \otimes \cdots \otimes {\bm  I}_{n_1} \in \RR^{N \times N},
$$
where ${\bm  A}_{m} \in \RR^{n_m \times n_m}$ contains the approximation of the second derivative in the $x_m$ direction. We will also consider problems where the nonlinear term $\bm{f}$ depends on the first derivative of the state vector, so that we also define the matrix

$$
{\bm  D} = \sum_{m = 1}^{d} {\bm  I}_{n_d} \otimes \cdots \otimes \overset{m}{\bm  B}_{m} \otimes \cdots \otimes {\bm  I}_{n_1} \in \RR^{N \times N},
$$
where ${\bm  B}_{m} \in \RR^{n_m \times n_m}$ contains the approximation of the first derivative in the $x_m$ direction.
The vectors ${\bm y}(t) \in \RR^N$ from \eqref{eq2} then represent the vectorization of the elements of a tensor $\pmb{\mathcal{Y}}(t) \in \RR^{n_1 \times \cdots \times n_d}$, such that ${\bm y}(t)  = {\tt vec}(\pmb{\mathcal{Y}}(t))$, ${\bm L}{\bm y} = \mbox{vec}\left(\mathcal{A}(\pmb{\mathcal{Y}})\right)$ and ${\bm  D}{\bm  y} = \mbox{vec}\left(\mathcal{D}(\pmb{\mathcal{Y}})\right)$, where\footnote{For the case $d = 2$, \eqref{A3} are Sylvester operators of the form ${\bm  A}_{1}{\bm Y} + {\bm Y}{\bm  A}_{2}^{\top}$ and ${\bm  B}_{1}{\bm Y} + {\bm Y}{\bm B}_{2}^{\top}$ respectively \cite{Simoncini2017}.}
\begin{equation}\label{A3}
\mathcal{A}(\pmb{\mathcal{Y}})  := \sum_{m = 1}^d\, \pmb{\mathcal{Y}} \times_m {\bm  A}_{m} \quad \mbox{and} \quad \mathcal{D}(\pmb{\mathcal{Y}})  := \sum_{m = 1}^d\, \pmb{\mathcal{Y}} \times_m {\bm  B}_{m}.
\end{equation}
Moreover, if the function $\mathcal{F}: {\cal S} \times [0,t_f] \rightarrow  \RR^{n_1 \times \cdots \times n_d}$ represents the function ${\bm  f}$ evaluated at the entries of the array $\pmb{\mathcal{Y}}$ and $\mathcal{D}(\pmb{\mathcal{Y}})$, then it holds that ${\bm  f}({\bm D}{\bm y}, {\bm y}(s),u(s),s) = {\tt vec}\left(\mathcal{F}\left(\mathcal{D}(\pmb{\mathcal{Y}}),\pmb{\mathcal{Y}},u(s), s\right)\right)$, and \eqref{eq2} can be written in the form
\begin{equation}\label{arraybigsystem}
\begin{cases}
\dot{  \pmb{\mathcal{Y}}  }(s) &= \mathcal{A}(\pmb{\mathcal{Y}}(s)) + \mathcal{F}\left(\mathcal{D}(\pmb{\mathcal{Y}}(s)),\pmb{\mathcal{Y}}(s),u(s), s\right),\\
{  \pmb{\mathcal{Y}}  }(t) &=  \pmb{\mathcal{X}} \in \RR^{n_1 \times \cdots \times n_d}.\\
\end{cases}
\end{equation}
The boundary conditions are contained in the matrices ${\bm  A}_{m}$ and ${\bm  B}_{m}$, $m = 1,\ldots,d$; see e.g., \cite{Autilia2019matri, palitta2016}. From here on forward we consider the case where $n_1 = \cdots = n_d = n$, so that $N = n^d$.

\subsubsection{Higher-Order POD (HO-POD) model reduction}
\label{secHOPOD}
As it has been shown in \cite{kirsten.22}, great savings in terms of memory requirements and computational time can be obtained by applying model order reduction directly to the system \eqref{arraybigsystem} instead of first vectorizing and applying model reduction to the vectorized system \eqref{eq2}. To this end, we consider an approximation of the form
$$
\pmb{\mathcal{Y}}(s) \approx \widetilde{\pmb{\mathcal{Y}}}(s) := \widehat{\pmb{\mathcal{Y}}}(s) \bigtimes_{m=1}^d {\bm  V}_m,
$$
where ${\bm  V}_m \in \RR^{n_m \times k_m}$ are tall matrices with orthonormal columns and $\widehat{\pmb{\mathcal{Y}}}(s) \in \RR^{k_1 \times \cdots \times k_d}$ $(k_m \ll n)$ satisfies the low-dimensional equation
\begin{equation}\label{arraysmallsystem}
\begin{cases}
\dot{ \widehat{ \pmb{\mathcal{Y}}}  }(s) &= \widehat{\mathcal{A}}(\widehat{ \pmb{\mathcal{Y}}}(s)) + \widehat{\mathcal{F}}\left(\widehat{\mathcal{D}}(\widehat{\pmb{\mathcal{Y}}}(s)),\widehat{\pmb{\mathcal{Y}}}(s),u(s), s\right),\\
\widehat{  \pmb{\mathcal{Y}}  }(t) &=  \pmb{\mathcal{X}} \bigtimes_{m=1}^d {\bm  V}_m^{\top} \in \RR^{k_1 \times \cdots \times k_d},\\
\end{cases}
\end{equation}
where 
\begin{equation}
\label{fbottle}
\widehat{\mathcal{F}}\left(\widehat{\mathcal{D}}(\widehat{\pmb{\mathcal{Y}}}(s)),\widehat{\pmb{\mathcal{Y}}}(s),u(s), s\right) = {\mathcal{F}}\left({\mathcal{D}}(\widetilde{\pmb{\mathcal{Y}}}(s)),\widetilde{\pmb{\mathcal{Y}}}(s),u(s), s\right) \bigtimes_{m=1}^d {\bm  V}_m^{\top}
\end{equation}
 and
\begin{equation}
\label{ABC}
\widehat{\mathcal{A}}(\widehat{\pmb{\mathcal{Y}}})  := \sum_{m = 1}^d\, \widehat{\pmb{\mathcal{Y}}} \times_m \widehat{\bm  A}_{m}, \quad \widehat{\bm  A}_m = {\bm  V}_{m}^{\top}{\bm  A}_m{\bm  V}_{m}.
\end{equation} 
The matrices ${\bm V}_m \in \RR^{n_m \times k_m}$ can be obtained via the HO-POD algorithm described in \cite{kirsten.22}. That is, given a set of snapshots $\{\pmb{\mathcal{Y}}(s_i)\}_{i = 1}^{n_s}$, each matrix ${\bf V}_m$ is constructed in order to approximate the left range space of the matrix 
$$
\pmb{\mathcal{S}}_{(m)} = \begin{pmatrix} \pmb{\mathcal{Y}}_{(m)}(s_1),\ldots,\pmb{\mathcal{Y}}_{(m)}(s_{n_s})\end{pmatrix} \in \RR^{n \times n^{d-1}n_s}, \quad \mbox{for} \quad m = 1,\ldots,d,
$$
where $m$ represents the mode along which the tensor is unfolded, and $\pmb{\mathcal{S}} \in \RR^{n \times \cdots \times n}$ is a tensor of order $d$ containing the snapshots. Note that neither the matrices $\pmb{\mathcal{S}}_{(m)}$ or the tensor $\pmb{\mathcal{S}}$ is ever explicitly constructed or stored. Instead we follow the dynamic algorithm initially introduced in \cite{Kirsten.Simoncini.arxiv2020} for approximating the left range space of $\pmb{\mathcal{S}}_{(m)}$. In \cite{kirsten.22} a simpler algorithm was used to construct the approximation space in the tensor setting. Here we implement the more refined dynamic algorithm for the tensor setting; the inclusion of snapshot information is discussed here, whereas snapshot selection will be presented in \cref{sec:hints}.

Suppose $\kappa$\footnote{We refer the reader to \cite{Kirsten.Simoncini.arxiv2020} for a detailed experimental analysis on the role of the parameter $\kappa$.} is the maximum admissible dimension for the reduced space in all modes, selected a-priori, and consider the initial condition $\pmb{\mathcal{Y}}(t)$. 
Let
$$
\pmb{\mathcal{Y}}(t) \approx \pmb{\mathcal{C}}(t) \bigtimes_{m = 1}^d  {\bm U}_{m}^{(0)},
$$
represent the sequentially truncated higher order SVD\footnote{For the case $d = 2$, however, we just use the standard MATLAB {\tt SVD} function.} (STHOSVD) \cite{vannieuwenhoven2012} of $\pmb{\mathcal{Y}}(t)$, where ${\bm U}_{m}^{(0)}$ contains the first $\kappa$ dominant left singular vectors of $\pmb{\mathcal{Y}}_{(m)}(t)$. For each mode these left singular vectors are collected into the matrix $\widetilde{\bm V}_m = {\bm U}_{m}^{(0)}$, $m = 1,\ldots, d$. 

Subsequently, suppose the snapshot at time instance $s_j$ has been selected for inclusion into the approximation space and let 
$
\pmb{\mathcal{Y}}(s_j) \approx \pmb{\mathcal{C}}(s_j) \bigtimes_{m = 1}^d  {\bm U}_{m}^{(j)},
$
represent the STHOSVD of the selected snapshot and let ${\bm \Sigma}_m^{(j)}$ contain the first $\kappa$ singular values of $\pmb{\mathcal{Y}}_{(m)}(s_j)$ on the main diagonal. The approximation spaces are updated by appending the new singular values and vectors, so that 
$$
\widetilde{\bm V}_m \leftarrow [\widetilde{\bm V}_m, {\bm U}_{m}^{(j)}], \quad {\rm and} \quad \widetilde{\bm \Sigma}_m \leftarrow {\tt blkdiag}( \widetilde{\bm \Sigma}_m,  {\bm \Sigma}_m^{(j)}).
$$
Eventually the diagonal entries of $\widetilde{\bm \Sigma}_m$ are reordered decreasingly and truncated so that the largest $\kappa$ values are retained, with the vectors in $\widetilde{\bm V}_m$ reordered and truncated accordingly. 

At the end of the procedure, when all snapshots have been processed, the final basis vectors are obtained by orthogonal reduction of the matrices $\widetilde{\bm V}_m$. More precisely, let $\widetilde{\bm V}_m = \overline{\bm V}_m\,\overline{\bm \Sigma}_m\,\overline{\bm W}_m^\top$ be the SVD of $\widetilde{\bm V}_m$. The final basis matrices ${\bm V}_m$ are obtained by truncating the first $k_m$ dominant singular vectors of $\overline{\bm V}_m$ according to the criterion
\begin{equation} \label{kselect}
{\sqrt{\sum_{i = k_m + 1}^{\kappa} ({\sigma}_m^{(i)})^2}} < \tau\, {\sqrt{\sum_{i = 1}^{\kappa} ({\sigma}_m^{(i)})^2}},
\end{equation}
for some $\tau \in (0,1)$, where ${\sigma}_m^{(i)}$ is the $i$-th diagonal element of $\overline{\bm \Sigma}_m$.
\subsubsection{Higher-Order DEIM (HO-DEIM) approximation of the nonlinear term}
\label{secHODEIM}
It is clear from \eqref{fbottle} that a bottleneck forms around the reduced nonlinear term, similar to the vector setting. To this end, we consider the HO-DEIM algorithm from \cite{kirsten.22} to circumvent the issue. Consequently, suppose the tall matrices ${\bm \Phi}_m \in \RR^{n_m \times p_m}$, for $m = 1,2,\ldots,d$, have been constructed as the output of the HO-POD method described above for the snapshots $\{\mathcal{F}\left(\mathcal{D}(\pmb{\mathcal{Y}}(s_i)),\pmb{\mathcal{Y}}(s_i),u(s_i), s_i\right)\}_{i = 1}^{n_s}$. Furthermore, consider $d$ matrices ${\bm  P}_m \in \RR^{n_m \times p_m}$ each containing a subset of columns of the $n_m \times n_m$ identity matrix. The matrices ${\bm  P}_m$ are each respectively obtained as the output of the {\tt q-deim} algorithm \cite{gugercin2018} with input ${\bm \Phi}_m$. The {\sc ho-deim} approximation of \ref{fbottle} is then given by
\begin{equation}\label{tensordeim}
\begin{split}
\widehat{\mathcal{F}}\left(\widehat{\mathcal{D}}(\widehat{\pmb{\mathcal{Y}}}(s)),\widehat{\pmb{\mathcal{Y}}}(s),u(s), s\right) &\approx {\mathcal{F}}\left({\mathcal{D}}(\widetilde{\pmb{\mathcal{Y}}}(s)),\widetilde{\pmb{\mathcal{Y}}}(s),u(s), s\right)\bigtimes_{m = 1}^d {\bf F}_{m}\\
&= \widehat{\mathcal{F}}^{\mbox{\tiny \sc deim}}\left(\widehat{\mathcal{D}}(\widehat{\pmb{\mathcal{Y}}}(s)),\widehat{\pmb{\mathcal{Y}}}(s),u(s), s\right),
\end{split}
\end{equation}
where
$$
{\bf F}_{m} = {\bf V}_{m}^{\top}{\bf \Phi}_{m}({\bf P}_{m}^{\top}{\bf \Phi}_{m})^{-1}{\bf P}_m^{\top}.
$$
If $\mathcal{F}$ is evaluated element-wise at the components of $\widetilde{\pmb{\mathcal{Y}}}(s)$ and ${\mathcal{D}}(\widetilde{\pmb{\mathcal{Y}}}(s))$, then it holds that
\begin{equation}\label{tensoreq}
\begin{split}
\reallywidehat{{\mathcal{F}}\left({\mathcal{D}}(\widetilde{\pmb{\mathcal{Y}}}(s)),\widetilde{\pmb{\mathcal{Y}}}(s),u(s), s\right)} :=& {\mathcal{F}}\left({\mathcal{D}}(\widetilde{\pmb{\mathcal{Y}}}(s)),\widetilde{\pmb{\mathcal{Y}}}(s),u(s), s\right)\bigtimes_{m = 1}^d {\bf P}_{m}^{\top}\\ =& {\mathcal{F}}\left({\mathcal{D}}(\widetilde{\pmb{\mathcal{Y}}}(s)) \bigtimes_{m = 1}^d {\bf P}_{m}^{\top},\,\,\widetilde{\pmb{\mathcal{Y}}}(s) \bigtimes_{m = 1}^d {\bf P}_{m}^{\top},\,u(s), s\right).
\end{split}
\end{equation}
In this case the nonlinear term is evaluated at only $p_1p_2\cdots p_m$ entries.

\subsubsection{The reduced optimal control problem on a tree structure}

In this section we explore how the full procedure combining the tree structure algorithm and the HO-POD-DEIM Model reduction technique is split into an \textit{offline} and \textit{online} stage to solve the HJB equation \eqref{HJB} and determine the optimal control \eqref{feedback}.

\paragraph{\bf Offline Stage} 
The offline stage consists of two important steps, namely \textit{snapshot collection} and \textit{basis construction}. To this end, we select a coarse time step $\widehat{\Delta t}$ and control set $\widehat{U}$. The basis is constructed on the fly following Sections \ref{secHOPOD}-\ref{secHODEIM}, on the nodes of the tree, which is constructed following Section \ref{sec2.1}. This is a computationally expensive step, as the full-dimensional space is explored in this phase. To this end, we discuss a collection of nuances related to the implementation in Section \ref{sec:hints}. 
\paragraph{\bf Online Stage} 
At this stage, the computed basis vectors are exploited to construct a reduced dimensional tree, approximate the reduced value function and compute the optimal trajectory at a fraction of the inital cost.
\begin{itemize}
\item{\bf Construction of the reduced tree.} 
Here we fix the desired wider discrete control set $\widetilde{U} \subset \widehat{U}$ and/or a smaller time step $\Delta t \le \widehat{\Delta t}$ for the resolution of the HO-POD-DEIM reduced dynamical system 
\begin{equation}\label{fullsmallsystem}
\begin{cases}
\dot{ \widehat{ \pmb{\mathcal{Y}}}  }(s) &= \widehat{\mathcal{A}}(\widehat{ \pmb{\mathcal{Y}}}(s)) + \widehat{\mathcal{F}}^{\mbox{\tiny \sc deim}}\left(\widehat{\mathcal{D}}(\widehat{\pmb{\mathcal{Y}}}(s)),\widehat{\pmb{\mathcal{Y}}}(s),u(s), s\right),\\
\widehat{  \pmb{\mathcal{Y}}  }(t) &=  \pmb{\mathcal{X}} \bigtimes_{m=1}^d {\bm  V}_m^{\top} \in \RR^{k_1 \times \cdots \times k_d}.\\
\end{cases}
\end{equation}
Following Section \ref{sec2.1}, we build the reduced tree $\widehat{\mathcal{T}}$ as done for the offline stage. The cardinality of tree, however, still grows exponentially, despite the reduced dimension of the dynamical system. As a result we analyze a collection of important pruning criteria in Section \ref{sec:pruning} in an attempt to reduce the cardinality of the tree dynamically during the construction.

\item{\bf Approximation of the reduced value function.} 
The value function computed in the reduced space will be denoted by $\widehat{V}(\widehat{x},t)$ and its approximation at time $t_n$ as $\widehat{V}^{n}(\widehat{x})$.
Its resolution will follow the classical scheme introduced in Section \ref{sec2.1} :

\begin{equation}
\begin{cases}
\widehat{V}^{n}(\widehat{\zeta}^{n}_i)= \min\limits_{u\in \widetilde{U}} \{ \widehat{V}^{n+1}(T^{n+1}(\widehat{\zeta}^{n}_i,u) +\Delta t \, \widehat{L}(\widehat{\zeta}^{n}_i,u,t_n) \}, \\
 \qquad \qquad\qquad \qquad\qquad  \widehat{\zeta}^{n}_i \in \widehat{\mathcal{T}}^{n}\,, n = N_t-1,\ldots, 0, \\
\widehat{V}^{N_t}(\widehat{\zeta}^{N_t}_i)= \widehat{g}(\widehat{\zeta}_i^{N_t}), \qquad\qquad \qquad\qquad \qquad\qquad   \widehat{\zeta}_i^{N_t} \in \widehat{\mathcal{T}}^{N_t},
\end{cases}
\label{HJB-2S}
\end{equation}

where
$$
\widehat{L}(\widehat{\zeta},u,t) =  L(\widehat{\zeta} \bigtimes_{m=1}^d {\bm  V}_m^{\top},u,t),
$$
$$
\widehat{g}(\widehat{\zeta}) = g(\widehat{\zeta}_i^{N_t}\bigtimes_{m=1}^d {\bm  V}_m^{\top})
$$
and $T^{n+1}(\widehat{\zeta},u)$ stands for the time evolution of the node $\widehat{\zeta}$ with control $u$ at time $t_{n+1}$.

\item{\bf Computation of the optimal trajectory.} 

The optimal trajectory can be seen as a specific path in the tree structure. For this reason during the computation of the value function we store the minimizing indices in \eqref{HJB-2S}. Once completed the computation of the value function, the optimal path will be given just following the tree branches which returns the minimum index.

\end{itemize}

\section{Hints for the implementation}
\label{sec:hints}

In both the offline and online phases of the procedure, great savings in terms of CPU time and memory requirements can be obtained if implemented in an efficient way. Consequently, in this section we discuss how the snapshots are selected and how the simulated tree nodes can be efficiently stored to save on memory requirements in the offline phase. Moreover, we discuss how the reduced model can be efficiently simulated at many time steps and control inputs, to avoid high computational costs in the online phase.

\subsection{Snapshot selection}
The full order model is simulated on a coarse timegrid with two control inputs which are the two extremes of the control domain, as discussed in \cite{Alla_Saluzzi_2020}. To avoid excessive computational work we only include information from snapshots that are not yet well approximated in the current basis. The condition for snapshot inclusion is given by the projection error, that is:
\begin{equation}
\label{snapcond}
\mbox{if:} \quad \frac{\| {\pmb{\cal Y}}(s_i) - {\pmb{\cal Y}}(s_i) \bigtimes_{m = 1}^d {\bm V}_m{\bm V}_m^\top \|_F}{\|{\pmb{\cal Y}}(s_i)\|_F} > \tau \quad \mbox{then:} \quad \mbox{\bf Include}, 
\end{equation} 
for some $\tau \in (0,1)$. Here the matrices ${\bm V}_m$ contain the current basis vectors in all modes, updated dynamically with snapshot information from the previous selected snapshots as described in \cref{secHOPOD}. 
\subsection{Efficient memory allocation by low-rank storage of tree nodes}
\label{mem}

One challenge of the method presented in \cite{Alla_Saluzzi_2020} is in terms of memory in the offline phase, since the high fidelity solutions need to be calculated and stored for several time steps and control inputs in order to form a reduced order model. In particular, the nodes of the tree $\mathcal{T}$ are vectorized and stored in a matrix ${\bm T} \in \mathbb{R}^{N \times |\mathcal{T}|}$, where $|\mathcal{T}|=O(M^{N_t})$, with $M$ fixed as the number of control inputs and $N_t$ as the number of time steps. The exponential growth of the second dimension greatly limits the number of snapshots that can be stored. 

In this paper, we suggest the following improvement. Since the full order model is simulated in array form, the snapshots at the resulting tree nodes are either matrices or tensors. Consequently, we can take advantage of the (possible) low-rank structure of each node. That is, we compute the STHOSVD of each computed nodal value, truncated to the first $\kappa$ singular vectors in each mode, so that $
\pmb{\mathcal{Y}}(s_j) \approx \pmb{\mathcal{C}}(s_j) \bigtimes_{m = 1}^d  {\bm U}_{m}^{(j)}
$,  with $\kappa$ selected a-priori as discussed in \cref{secHOPOD}. As a result, the node can be stored in low-rank form to be recalled for later computations. More precisely, we collect and store only the dominant  singular vectors in each mode and the low-dimensional core tensors such that
$$
\overline{\bm U}_{m} \leftarrow [\overline{\bm U}_{m}, {\bm U}_{m}^{(j)}] \quad {\rm for}\,\,\, m = 1,\ldots,d \,\,\, {\rm and} \quad \overline{\bm c} \leftarrow [\overline{\bm c}, {\rm vec}(\pmb{\mathcal{C}}(s_j))].
$$
When required at the next time level, the snapshots can easily be computed from its Tucker decomposition \cite[Section 4]{kolda2009}.
This process allows us to store vectors of length $n_1,\ldots,n_d$ instead of $N = n_1n_2\cdots n_d$. The number of vectors stored depends on the rank of the considered snapshots. A further advantage is that when a snapshot is selected for inclusion into the approximation space, a HOSVD is required as discussed in \cref{secHOPOD}, which will be readily available thanks to this procedure.

Furthermore, it has been observed that only the nodes from the previous level of the tree need to be stored, since the snapshots from the earlier levels are automatically processed and discarded during the HO-POD basis construction. Finally, we observe that the computation of the value function does not require the knowledge of the nodes, but only of the corresponding cost evaluation. In this way we are going to store only the corresponding scalar cost and the nodes will be erased after the computation of its tree sons.

\subsection{Efficient simulation of the reduced model \eqref{fullsmallsystem}}
An important ingredient in the success of the HO-POD-DEIM reduction procedure is the ability to integrate the reduced model \eqref{fullsmallsystem} in array form without vectorization. In this paper we consider the semi-implicit Euler scheme, given that the considered model is typically associated with a stiff linear term and a nonstiff nonlinear term, but several alternatives can be considered \cite{Autilia2019matri, kirsten.22}. More precisely, suppose ${\widehat{\pmb{\cal Y}}}^{(j)}$ is an approximation of ${\widehat{\pmb{\cal Y}}}(s_j)$, then the linear system
\begin{equation}\label{tensorlinearsystem}
    (\widehat{{\cal I}} - \Delta t \widehat{\cal{A}}){\widehat{\pmb{\cal Y}}}^{(j)} = {\widehat{\pmb{\cal Y}}}^{(j - 1)} + \Delta t \widehat{\mathcal{F}}^{\mbox{\tiny \sc deim}}\left(\widehat{\mathcal{D}}({\widehat{\pmb{\cal Y}}}^{(j - 1)}),{\widehat{\pmb{\cal Y}}}^{(j - 1)},u(s_{j-1}), s_{j-1}\right)
\end{equation}
needs to evaluated at each time level $s_j$. Once again, vectorizing the linear system \eqref{tensorlinearsystem} at each time step will reduce the computational gains related to the reduction in Array form. Instead, \eqref{tensorlinearsystem} can be solved in array form using the direct method presented in \cite{simoncini.2020.boll,kirsten.22}.
\section{Pruning techniques}
\label{sec:prun}
Although theoretically the tree structure enables to compute the solution for arbitrary high dimensional problems, since we are not restricted to the direct discretization of a domain, its construction turns to be computationally expensive, due to the exponential growth of its cardinality,
For this reason in this section we are going to introduce and analyse different $pruning$ $criteria$ able to reduce the growth of the tree, but keeping the same accuracy.

\subsection{Geometric pruning}
\label{sec:pruning}
A pruning criterion based on a comparison of the nodes in euclidean norm has been introduced in \cite{saluzzi2022error}. More precisely, two given nodes $\zeta^n_i$ and $\zeta^n_j$ will be merged if 
\begin{equation}
\Vert \zeta^n_i-\zeta^n_j \Vert \le \epsilon, \quad \mbox{ with }i\ne j \mbox{ and } n = 0,\ldots, N_t, 
\end{equation}
for a given threshold $\epsilon>0$. To ensure first order convergence, the threshold $\epsilon$ must scale as $\Delta t^2$ (we refer to \cite{saluzzi2022error} for more details about the error estimates).
This pruning criterion has been successfully applied to low and high dimensional problems, but the main drawback is the expensive computation of distances in high dimension.
One possible solution relies on the projection of the data onto a lower dimension minimizing the variance of the data.
This procedure is already encoded in the above described algorithm, since we are reducing the dimension of the problem keeping the main features. 
\subsection{Statistical pruning}
\label{stat_pru}
In this section we introduce a new iterative pruning criterion based on statistical information about the value function.
We suppose we are starting with a certain control set $U_1$. First, we construct the tree $\mathcal{T}_1$ based on the control set $U_1$ and the value function computed on the tree will be denoted by $V_1(x,t)$.
Afterwards, we refine the constructed tree based on the information on the value function: fixing a ratio $\rho \in (0,1]$ of the nodes, we retain just those with the lowest value function, obtaining a new tree $\widetilde{\mathcal{T}}_1$. More precisely, we have that $|\widetilde{\mathcal{T}}_1|=\rho |{\mathcal{T}}_1|$ and for every time level $t_n \in [t,T]$ and every node $\zeta \in \mathcal{T}_1$ there exists a node $\widetilde{\zeta} \in \widetilde{\mathcal{T}}_1$ such that $V_1(\widetilde{\zeta},t_n) \le V_1(\zeta,t_n)$.
Hence, we can start with the construction of a new tree $\mathcal{T}_2$ with a wider control set $U_2 \supset U_1$ such that the nodes are constrained in the zones where the previous value function had the lowest values, $i.e.$
    \begin{equation*}
        \min_{\widetilde{\mathcal{T}}^n_1} \widetilde{\zeta}  \le \zeta^n_j \le  \max_{\widetilde{\mathcal{T}}^n_1}\widetilde{\zeta} , \quad \forall \zeta^n_j \in \mathcal{T}^n_2, \, n \in \{N_{start},\ldots, N_t \} 
     \end{equation*}
where the minimum and the maximum are computed element-wise, as well the inequalities. The statistical pruning is applied starting from an arbitrary time $t_{N_{start}}$ since the first levels contain few nodes. We usually will fix $N_{start}=3$. The entire procedure can be iterated doubling the number of controls in each step. Computed the tree $\widetilde{\mathcal{T}}_k$ at the $k$-th iteration, the subsequent tree $\mathcal{T}_{k+1}$ will satisfy the constraint
    \begin{equation}
        \min_{\widetilde{\mathcal{T}}^n_{k}} \widetilde{\zeta}  \le \zeta^n_j \le  \max_{\widetilde{\mathcal{T}}^n_k}\widetilde{\zeta} , \quad \forall \zeta^n_j \in \mathcal{T}^n_{k+1}, \, n \in \{N_{start},\ldots, N_t \} .
        \label{statistical}
     \end{equation}
Since we neglect the nodes which do not satisfy \eqref{statistical}, the problem can be regarded as a state-constrained problem where the constraint is given by the relation \eqref{statistical}. We refer to \cite{afs20} for more details about the coupling of the tree with state-constrained problems. 
The ratio $\rho$ is fixed such that it still retains the optimal trajectory from the previous iteration. In this way we can ensure that the value function is not increasing during the iterative procedure. Therefore, we denote by $V^n_k(x)$ the value function obtained in the $k$-th iteration at the point $x$ at time $t_n$. By construction, we can notice that the iterative value function at the initial time is non increasing, $i.e.$
$$
V^0_{k+1}(x) \le V^0_{k}(x), \quad \forall k \ge 0
$$
and bounded from below since
$$
V^n_{k}(x) \ge -T M_f - M_g, \quad \forall k \ge 0,\, n \in \{0, \ldots, N_t\}, \, x \in \mathcal{T}^n_{k},
$$
using the hypothesis \eqref{Mf}.
Hence, the iterative scheme is convergent and it can repeated until we reach a maximum number of iterations or it satisfies a stopping criterion. In Algorithm \ref{alg_1} the method is sketched.

\begin{algorithm}[H]
\caption{Statistical pruning}
\begin{algorithmic}[1]
\State Choose an initial condition $x_0$, a ratio $\rho$, a starting time $t_{N_{start}}$, a tolerance $tol$, a maximum number of iteration $k_{max}$ and a initial number of discrete controls $M$
\State Build a tree $\mathcal{T}_1$ with $M$ controls and compute the value function $V^n_1(x)$
\While{$res > tol$ and $k \le k_{max}$}
\State{$M:=2M-1$}
\State{Construct $\widetilde{\mathcal{T}}_{k}$ retaining a ratio $\rho$ of $\mathcal{T}_{k}$ with the lowest value function} 
\State{Construct $\mathcal{T}_{k+1}$ under the constraint \eqref{statistical}}
\State{Compute the value function $V^n_{k+1}(x)$}
\State{$res=|V^0_{k}(x_0)-V^0_{k+1}(x_0)|$}\label{alg_1_res}
\State{$k=k+1$}
\EndWhile
\end{algorithmic}
\label{alg_1}
\end{algorithm}
 In Figure \ref{statistical_pru} we show an application of the statistical pruning under the Van der Pol dynamics:
\begin{equation*}
\left\{ \begin{array}{l}
\dot{y_1}(t)=y_2(t), \;\; t\in(0,T],\\
\dot{y_2}(t)=\omega(1-y^2_1(t))y_2(t)-y_1(t)+u(t), \;\; t\in(0,T],\\
(y_1(0),y_2(0)) = (0.4,-0.3),
\end{array} \right.
\end{equation*}
where $\omega = 0.15$, $T=1.4$ and $u:[0,T] \rightarrow [0,1].$
 Fixing a time step $\Delta t = 0.2$, we display the initial full tree $\mathcal{T}_1$ with discrete controls $\{0,1\}$, its refinement $\widetilde{\mathcal{T}}_1$ with $\rho=0.3$ and the new tree $\mathcal{T}_2$ with discrete controls $\{0,0.5,1\}$.

\begin{figure}[htb!]
\centering
  \includegraphics[scale=0.5]{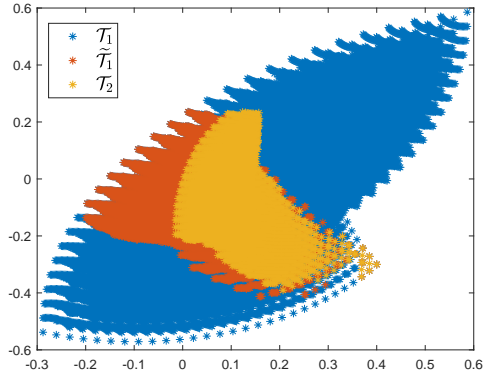}
 \caption{Application of the statistical pruning to Van der Pol oscillator with $\rho=0.3$.}
 \label{statistical_pru}
 \end{figure}

\subsection{Monotone control}
\label{monotone}
In this section we restrict the admissible set of controls to monotone controls, $e.g.$ 
$$
 \overline{\mathcal{U}}=\{u:[0,T] \rightarrow U, \,u(\cdot) \mbox{ monotone in }[0,T] \}. 
 $$
 In economy different problems can be formulated as optimal control problems with
monotone controls ($e.g.$ adjustment theory of investment problems).
 Under this constraint, in \cite{barron1985viscosity} Barron proved that the value function is a generalized solution of the quasi-variational inequality and its numerical treatment has been investigated in a series of papers \cite{philipp2015discrete,aragone2018fully}.
We consider the non decreasing case without loss of generality. 
Let us introduce the notation which will be useful in this section.
We define as $\mathcal{T}_{M,N}$ the tree obtained using $M$ discrete controls and $N$ time steps, while we denote as $\overline{\mathcal{T}}_{M,N}$ the tree constructed via monotone controls. In Figure \ref{fig:tree} we show the structure of the tree $\overline{\mathcal{T}}_{2,N}$. In this case we are using $2$ discrete controls $u_1<u_2$. When we apply $u_2$, the corresponding subtree will have just one node for each level, since the control cannot decrease by hypothesis.

\begin{figure}[htbp]
\centering
\begin{tikzpicture}
[->,level/.style={sibling distance=25mm/#1},  grow = right,]
  \node [circle,draw] {$x$}
     child { node [circle,draw] {$\zeta^1_{1}$} child{ node [circle,draw] {$\zeta^2_{3}$} child child  } child{ node [circle,draw] {$\zeta^2_{4}$} child { node {$\ldots$}}} }
         child { node [circle,draw] {$\zeta^1_{2}$} 
    child{ node {$\ldots$} }
    };
\end{tikzpicture}    
\caption{Example of the tree $\overline{\mathcal{T}}_{2,n}$}
\label{fig:tree}
\end{figure}
In this framework we have a great improvement in terms of the cardinality of the tree, as stated in the following proposition.

\vspace{2mm}

\begin{prop}
Given $M$ discrete controls and $N$ time steps, the cardinality of the tree based on monotone controls is given by
\begin{equation}
\left| \overline{\mathcal{T}}_{M,N}\right| = \frac{ (M+N)!}{M! N!}
\label{mon_card}
\end{equation}
\end{prop}

{\em Proof.}
We will proceed by induction on the pair $(M,N)$.
It is easy to check that $\left|\overline{\mathcal{T}}_{1,N}\right|= N+1$ and $\left|\overline{\mathcal{T}}_{M,1}\right|= M+1$. Now let us suppose \eqref{mon_card} holds for a pair $(M,N)$. First, we are going to prove that the result holds for the pair $(M,N+1)$.
Given the particular structure of the tree, we can write
$$
\left|\overline{\mathcal{T}}_{M,N+1}\right|= \sum_{k=1}^M \left|\overline{\mathcal{T}}_{k,N}\right|+1 =
\frac{(M + N +1)!- M! (N+1)! }{M!(N+1)!} +1 = \frac{(M + N+1)! }{M!(N+1)!},
$$
obtaining the result.
Afterwards, let us demonstrate it for the pair $(M+1,N)$.
In this case we can split the tree in the following way
$$
\left|\overline{\mathcal{T}}_{M+1,N}\right|=\left|\overline{\mathcal{T}}_{M,N}\right|+\left|\overline{\mathcal{T}}_{M+1,N-1}\right| = \sum_{k=2}^N \left|\overline{\mathcal{T}}_{M,k}\right| +\left|\overline{\mathcal{T}}_{M+1,1}\right|
$$
$$
= \frac{ (M + N +1)!-(M+2)!  N! }{( M+1)! N!} + M+2= \frac{ (M + N +1)!}{( M+1)! N!}
$$
and this completes the proof.
$$
$$
$$
$$



In general the cardinality of the tree grows as $O(M^N)$ which is infeasible due to the huge amount of memory allocations. 
Fixing the number of discrete control $M$, the cardinality of the tree based on a monotone control grows as $O(N^M/M!)$, yielding an affordable algorithm for the computation of the optimal control with a high number of time steps.

\subsection{Bilinear control}
\label{bilinear_section}
Let us consider the following bilinear dynamical system:
\begin{equation}
\dot{y}=Ly + uy, \quad u \in U \subset \mathbb{R}.
\label{bilinear}
\end{equation}
Discretizing \eqref{bilinear} via a semi-implicit scheme, we obtain
\begin{equation}
y^{n}=(I-\Delta t L)^{-1} y^{n-1} (1+ \Delta t u^{n-1}) = (I-\Delta t L)^{-n} y_0 \prod_{i=0}^{n-1} (1+ \Delta t u^{i}).
\label{bilinear_discrete}
\end{equation}
Let us consider  now a new evolution $\tilde{y}^{n}$ of the discrete scheme at time $t_{n}$ with controls $\{\tilde{u}^i\}_{i=0}^{n-1}$. Then the distance between the two dynamics is given by
\begin{equation}
\Vert y^n-\tilde{y}^n \Vert \le \Vert (I-\Delta t L)^{-n} y_0 \Vert \left| \prod_{i=0}^{n-1} (1+ \Delta t u^{i})- \prod_{i=0}^{n-1} (1+ \Delta t \tilde{u}^{i}) \right|.
\label{dist_bilinear}
\end{equation}
Let us introduce a definition which will be useful in this section.

\vspace{2mm}

\begin{definition}
A discrete system with discrete controls satisfies the sum-based pruning property if for every pair of vectors $(u^0,\ldots, u^{n-1})$ and $(\tilde{u}^0,\ldots, \tilde{u}^{n-1})$ such that 
\begin{equation}
\sum_{i=0}^{n-1} u^i=\sum_{i=0}^{n-1} \tilde{u}^i
\label{sum_prop}
\end{equation}
the corresponding discrete solution $y^n$ and $\tilde{y}^n$ satisfy the geometrical pruning rule, $e.g.$ $\Vert y^n - \tilde{y}^n \Vert \le C \Delta t^2$.
\end{definition}

\vspace{2mm}

This class of discrete systems benefits from an important improvement in terms of the cardinality of the corresponding tree, as stated in the following proposition. For the proof we refer to Proposition 3.12 in \cite{saluzzi2022error}. 

\vspace{2mm}

\begin{prop}
The cardinality of the tree based on a system with $M$ discrete controls and $N$ time steps satisfying the sum-based pruning property is at most $\frac{N(N-1)}{2}(M-1)+N+1$.
\end{prop}

\vspace{2mm}

The discrete dynamics \eqref{bilinear_discrete} with two discrete controls belongs to the class of the system verifying the sum-based pruning property as stated in the next proposition.

\vspace{2mm}

\begin{prop}
The discrete system \eqref{bilinear_discrete} satisfies the sum-based pruning with $2$ discrete controls. Hence, the cardinality of the corresponding tree is at most $\frac{N(N+1)}{2} +1$.
\end{prop}

\vspace{2mm}

\begin{proof}
Let us consider two pair of vectors $(u^0,\ldots, u^{n-1})$ and $(\tilde{u}^0,\ldots, \tilde{u}^{n-1})$ verifying the sum-based pruning property \eqref{sum_prop}. Since we are considering two discrete control $(u_1,u_2) \in U \times U$, the upper bound for distance \eqref{dist_bilinear} between the two corresponding dynamics becomes
$$
 \Vert (I-\Delta t L)^{-n} y_0 \Vert \left|  (1+ \Delta t u_1)^{k_1} (1+ \Delta t u_2)^{n-k_1}- (1+ \Delta t u_1)^{k_2} (1+ \Delta t u_2)^{n-k_2}\right|
$$
with $k_1, k_2 \in \{0,\ldots,n\}$. By property \eqref{sum_prop} we immediately see that either $k_1=k_2$ or $u_1=u_2$, which implies that the two corresponding solutions coincide.
\end{proof}

In this case we can directly construct the tree based on this structure, without implementing any pruning criterion. The construction of the tree based on two discrete controls may be used as a fast and cheap procedure to get information about the full dimensional system. Once constructed the basis and projected the system onto the lower dimensional space, it is possible to consider an higher number of discrete controls.

\section{An error bound for the multilinear HJB-POD-DEIM algorithm}

The aim of this section is to derive an error estimate for the approximation of value function with the HO-POD-DEIM algorithm applied to the tree structure. The main reference of this section is \cite{chat2012}, where the authors obtain a state space error bounds for the solutions of the reduced systems via a POD-DEIM approach and the application of an implicit scheme for the time integration. Following their proof, we are going to extend the result to semi-implicit schemes in our multilinear setting.

First of all, we consider the vectorized form of dynamical system \eqref{eq2}, whose semi-implicit discretization with stepsize $\Delta t$ and discrete controls $\{u^j\}_{j=0}^{N_{t}-1}$ reads
\begin{equation}
    \frac{y^j-y^{j-1}}{\Delta t} = L y^{j}+{\bm  f}(y^{j-1},u^{j-1},t^{j-1}), \qquad j=1,\ldots,N_t.
    \label{semi_implicit_full}
\end{equation}

Taking into account the basis in vector form
$$
V_{Y} =  V_d \otimes \cdots \otimes V_1, \quad
V_{F} =  {\bf \Phi_d} \otimes \cdots \otimes {\bf \Phi_1} ,
$$  
$$
\mathbb{P} = V_F \left(({\bf P_d} \otimes \cdots \otimes {\bf P_1})^\top V_F \right)^{-1} \left({\bf P_d} \otimes \cdots \otimes {\bf P_1}\right)^\top,
$$
the vectorized form of the semi-implicit scheme for the reduced dynamics \eqref{fullsmallsystem} reads
\begin{equation}
   \frac{\hat{y}^j-\hat{y}^{j-1}}{\Delta t} = V_{Y}^\top L \,  V_{Y} \hat{y}^{j} + V_{Y}^\top \mathbb{P}  {\bm  f}(V_{Y} \hat{y}^{j-1},u^{j-1},t^{j-1}), \qquad j=1,\ldots,N_t.
   \label{semi_implicit_red}
\end{equation}

Our aim is to prove an error estimate between full order scheme \eqref{semi_implicit_full} and reduced one \eqref{semi_implicit_red}.

For this purpose we introduce the logarithmic norm of matrix $A\in \mathbb{C}^{n \times n}$ defined as
\begin{equation}
\label{log_norm}
\mu(A)= \sup_{x \in \mathbb{C}^n \setminus \{0\}} \frac{Re<Ax,x>}{\Vert x \Vert^2}.
\end{equation}

The logarithmic norm plays an important role for the stability analysis for continuous and discrete linear dynamical systems. Indeed, it is possible to prove that $\Vert e^{tA} \Vert \le e^{t\mu(A)}$ $\forall t \ge 0$ (see $e.g.$ \cite{soderlind2006logarithmic}) and by this inequality we can state that the dynamical system is stable if $\mu(A) \le 0$. The definition of this norm will be fundamental in the treatment of the implicit part of the scheme, while the Lipschitz-continuity of {\bf f} will be employed for the estimation of the explicit part. In the following proposition we prove that the error between the full order model \eqref{semi_implicit_full} and the lifted reduced order model \eqref{semi_implicit_red} depends on the accuracy of the HO-POD and HO-DEIM basis. The proof can be found in Appendix \ref{appendixA}.

\vspace{2mm}

\begin{proposition}
\label{prop_estimate}
Given $\{y^k\}_{k=0}^{N_t}$ the solution of the \eqref{semi_implicit_full} and $\{\hat{y}^k\}_{k=0}^{N_t}$ solution of \eqref{semi_implicit_red} with controls $\{u^j\}_{j=0}^{N_t-1}$ and time step $\Delta t$ satisfying $\Delta t \, \mu(V_Y^\top L V_Y) <1$, then
\begin{equation}
 \sum_{k=0}^{N_t} |y^{k}-  V_Y \hat{y}^{k}|^2 \le C(T)\left( \mathcal{E}_y + \mathcal{E}_f \right)
 \end{equation}
 with
 $$
{\cal E}_y = \sum_{j = 0}^{N_t} |y^j - V_Y V_Y^{\top}y^j|^2, \quad
{\cal E}_f = \sum_{j = 0}^{N_t-1} | {\bm  f}(y^{j},t^{j},u^{j})- V_F V_F^{\top} {\bm  f}(y^{j},t^{j},u^{j})|^2.
$$
\end{proposition}

\vspace{2mm}

Finally, we are ready to prove a convergence result for the continuous value function $v(x,t)$, solution of the HJB equation \eqref{HJB}, and the discrete value function solution $\{\widehat{V}^{n}(\widehat{x})\}_n$, solution of the scheme \eqref{HJB-2S}. For this purpose, let us define the continuous version of the DDP for the full model 
\begin{equation}
\label{HJBt3}
\begin{cases}
V(x,s) = \min\limits_{u\in U} \{ V(x+(t_{n+1}-s) f(x,u,s), t_{n+1}) + (t_{n+1}-s) \, L(x,u,s) \}, \\
V(x,T) = g(x), \hspace{6cm} x \in \mathbb{R}^d, s \in [t_n,t_{n+1}),
\end{cases}
\end{equation}
and its reduced version which reads:
\begin{equation}
\label{HJBt4}
\begin{cases}
\widehat{V}(\widehat{x},s) = \min\limits_{u\in U} \{ \widehat{V}(\widehat{x}+(t_{n+1}-s) f^\ell(x^\ell,u,s), t_{n+1}) + +(t_{n+1}-s) \, \widehat{L}(\widehat{x},u,s) \}, \\
\widehat{V}(\widehat{x},T) = \widehat{g}(\widehat{x}), \hspace{6cm} \widehat{x} \in \mathbb{R}^\ell, s \in [t_n,t_{n+1}).
\end{cases}
\end{equation}
Given the exact value function $v(x,s)$  and its continuous reduced approximation $\widehat{V}(V_Y x,s)$, the following theorem provides an error estimate for the approximation of the HJB equation by the HO-POD-DEIM approach. The assumptions and the main procedure for the following result can be found in Theorem 5.1 in \cite{Alla_Saluzzi_2020}.

\vspace{2mm}

\begin{theorem}
Given $v(x,s)$ the solution of the HJB equation \eqref{HJB} and its reduced approximation $\widehat{V}(V_Y x,s)$ solution of the scheme \eqref{HJBt4}, the following estimate holds
\begin{equation}
| v(x,s)- \widehat{V}(V_Y x,s) | \le C(T)\left( \Delta t + \mathcal{E}_y + \mathcal{E}_f \right) .
\end{equation}
\end{theorem}

\begin{proof}
The proof follows closely the procedure adopted for Theorem 5.1 in \cite{Alla_Saluzzi_2020}. The only difference arises in the estimation of the projection error between the FOM and the lifted ROM solutions and in this case we apply Proposition \ref{prop_estimate} to obtain the result.
\end{proof}

\section{Numerical tests}

In this section we test the proposed technique in different frameworks. In the first numerical test we consider a bilinear advection-diffusion equation, comparing the vector and matrix cases for the construction of the reduced basis. The second test is devoted to a nonlinear reaction-diffusion PDE where we show the efficiency of the statistical pruning coupled with the MOR technique. Finally, in the third test we consider a more challenging problem: the control of the 3D viscous Burgers' equations. We use this final example to indicate the power of the proposed algorithm in terms of CPU time and memory requirements with respect to the vector construction of the problem.
The numerical tests are performed on a Dell XPS 13 with Intel Core i7, 2.8GHz and 16GB RAM. The codes are written in Matlab R2022a.

\subsection{Test 1: Advection-diffusion equation}

In the first numerical test we conside the following bilinear advection-diffusion equation:
\begin{equation}
\begin{cases}
\partial_s y + (c_1,c_2) \cdot \nabla y =  \sigma \Delta y +  y u(s) & (x,s) \in \Omega \times [0,T],
\\
y(x,s) =0 & (x,s) \in \partial \Omega \times [0,T], \\
y(x,0)=y_0(x) & x \in \Omega,
\end{cases}
\label{pde_1}
\end{equation}
with
$$
u \in \mathcal{U} = \{u:[0,T] \rightarrow U, u \, measurable\}.
$$
The aim of the optimal control is to drive the solution to the equilibrium $\overline{y}(x) \equiv 0$ and to this end we introduce the following cost functional:
\begin{equation*}
J_{y_0,t}(u) = \int_t^T \left(  \int_{\Omega} |y(x,s)|^2 dx +   |u(s)|^2 \right) ds + \int_{\Omega} |y(x,T)|^2  dx.
\end{equation*}
Since we are considering a bilinear optimal control problem, we can benefit of the results presented in Section \ref{bilinear_section}, $e.g.$ discretizing the control set with two discrete controls, the total cardinality of the tree is order $O(N_t^2)$.
We fix $T=1$, $U=[-3,-1]$, $y_0(x) = \max(2-\Vert x \Vert^2,0)$, $\Omega = [-5,5]^2$, $\widehat{\Delta t}= \Delta t = 0.05$, $\widehat{U} = \{-3,-1\}$, $c_1=0.5$, $c_2=0$ and $\sigma=0$. Later on we will discuss the behaviour of the algorithm considering different choices for the coefficients $c_1,c_2$ and $\sigma$. Furthermore, we impose $\tau = 10^{-4}$ for both methods to obtain the same projection error. In this setting the cardinality of the tree is $3321$.
In Table \ref{table_dim} we show the dimension of the basis varying the number of the grid points in each direction.
Since the system is driven along an axes, the HO-POD procedure requires more basis in one direction with respect to the other. We note that the maximum of the dimensions of the HO-POD basis is equal to the number of POD basis for any choice of $n$.

\begin{table}[bht]
\centering
\begin{tabular}{l|cccccc}
       
 $n$ &$101$  & $121$  & $141$ & $161$ & $181$ & $201$  \\ \hline
 POD & 7 & 7 & 7& 7 & 8 & 8 \\
 HO-POD & $(3,7)$  & $(3,7)$  & $(3,7)$ & $(3,7)$  & $(3,8)$  & $(3,8)$
 \end{tabular}
 \caption{Dim. of basis for POD and HO-POD varying the number of the grid points per dimension with $\tau = 10^{-4}$, $\Delta t = 0.05$ and two discrete controls.  \label{table_dim}}
\end{table}

In the top left panel of Figure \ref{figCPUoffline} we compare the CPU time for the offline phase for the POD and HO-POD algorithms.
As stated previously, HO-POD requires less storage and enables to treat with very high dimensional problems. In particular, we note a difference of almost two orders of magnitude between the POD and HO-POD offline stages for $n = 201$. In the top right and bottom panels of Figure \ref{figCPUoffline} a comparison of the computational times and cardinality of the pruned trees varying the number of discrete controls for the online phase and fixing $n=161$ is presented. The HO-POD is again performing better than the POD algorithm since the geometrical pruning turns to be more efficient in the HO-POD setting. Indeed, fixing $M= 7$ discrete controls, the cardinality of the HO-POD tree reaches almost order $10^5$, against the order $\approx 10^7$ for the POD tree and $10^{16}$ for the unpruned tree.

\begin{figure}[htb!]
\minipage{0.49\textwidth}
  \includegraphics[width=\linewidth]{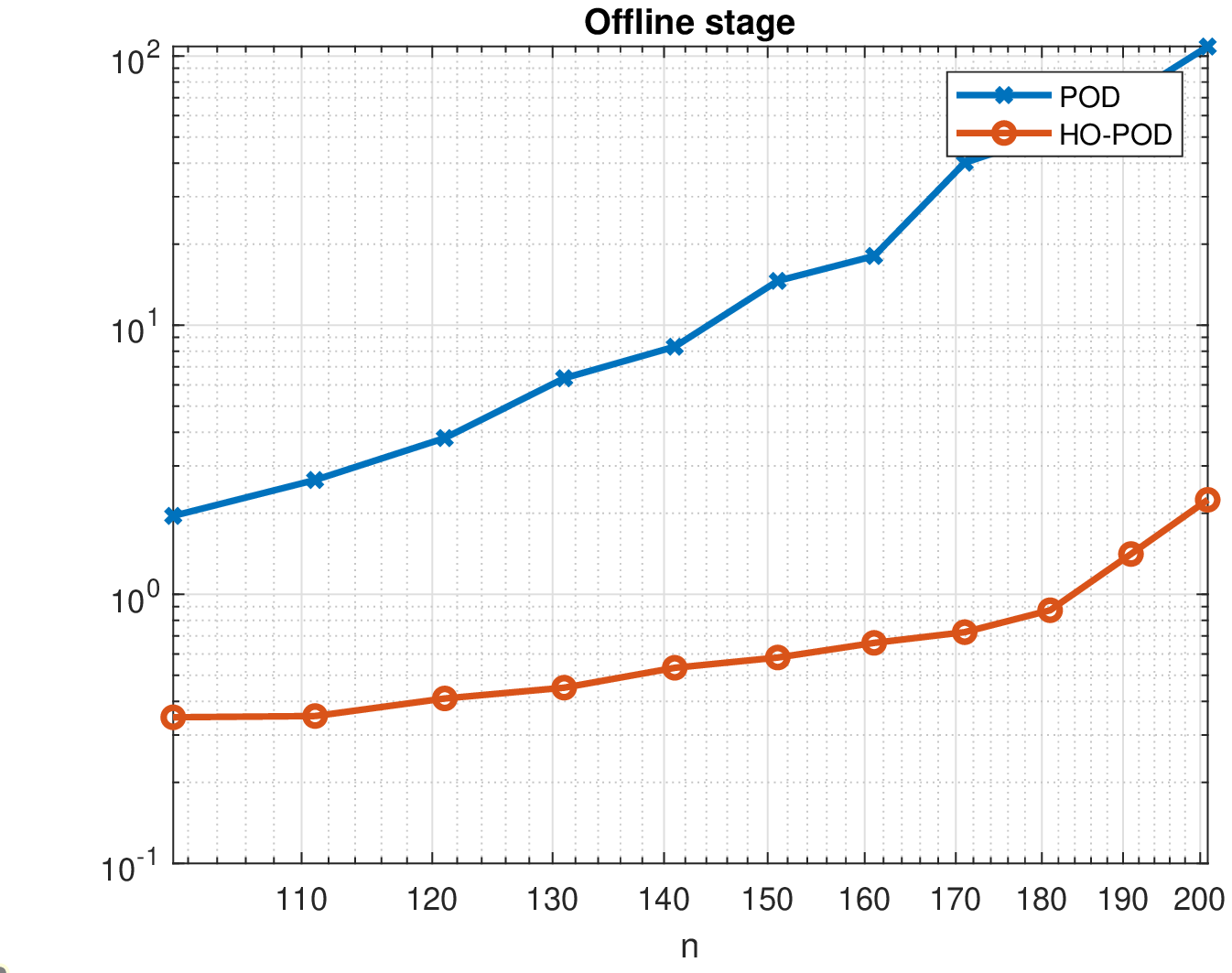}
\endminipage\hfill
\minipage{0.49\textwidth}
  \includegraphics[width=\linewidth]{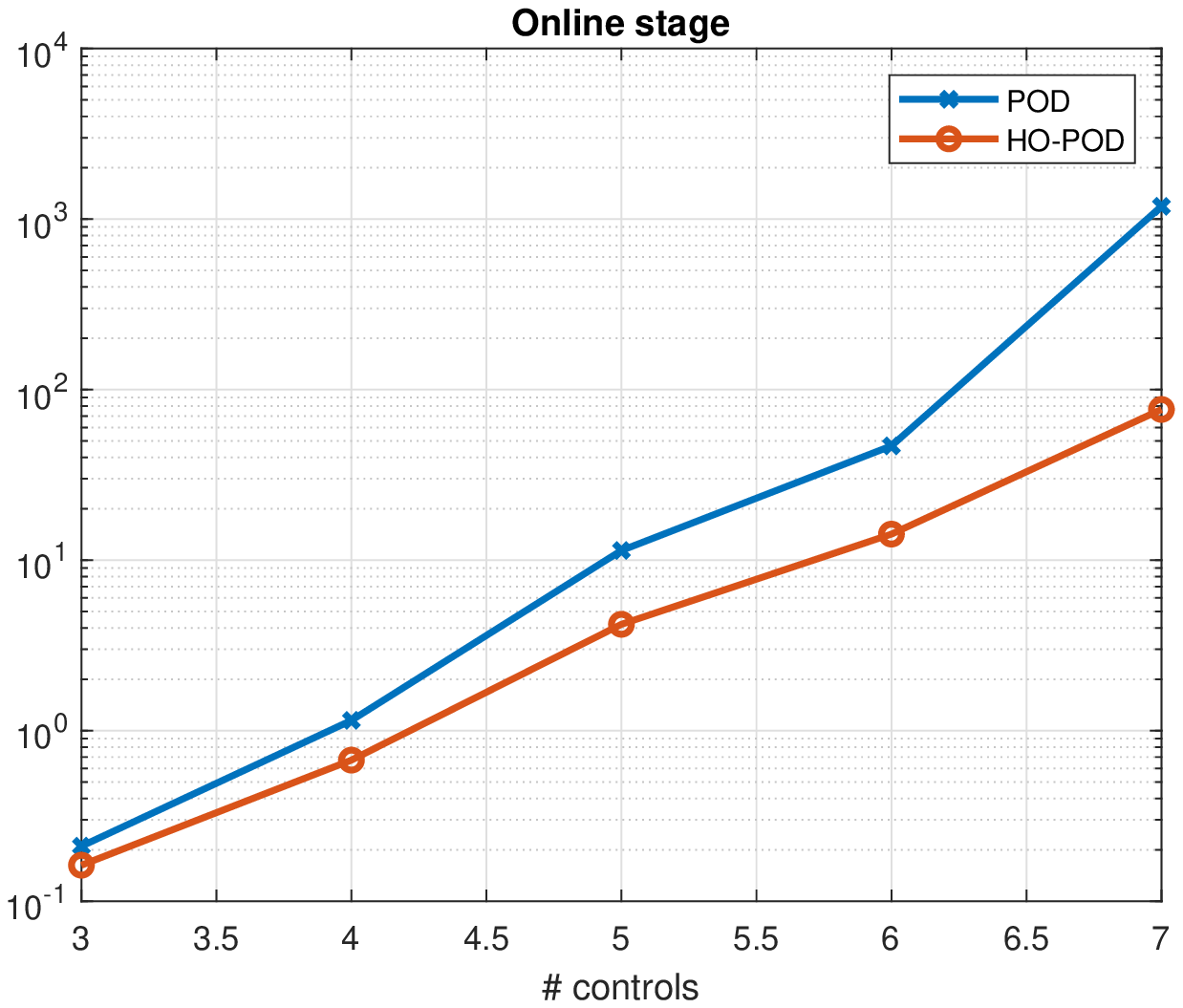}
\endminipage
\centering{
\minipage{0.49\textwidth}
  \includegraphics[width=\linewidth]{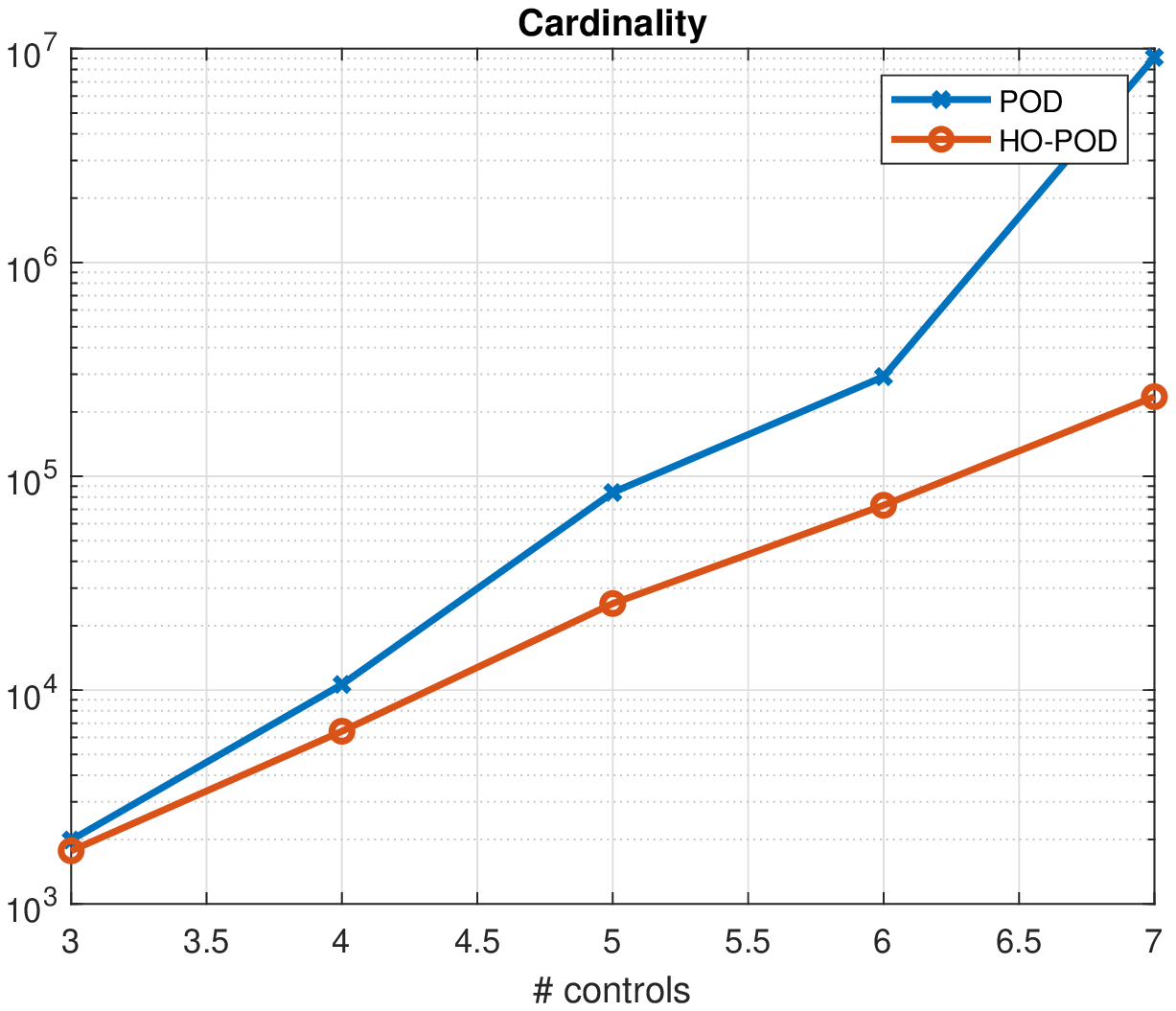}
\endminipage
}
\caption{Test 1: CPU time for the offline stage (top left), CPU time for the online stage (top right) and cardinality of the pruned tree (bottom) for POD and HO-POD techniques.}	
\label{figCPUoffline}
\end{figure}
The optimal trajectory computed via HO-POD with $7$ discrete controls for different time instances is displayed in Figure \ref{figtransport}, noting that the solution is getting closer to the stationary solution.

\begin{figure}[htb!]
\minipage{0.32\textwidth}
  \includegraphics[width=\linewidth]{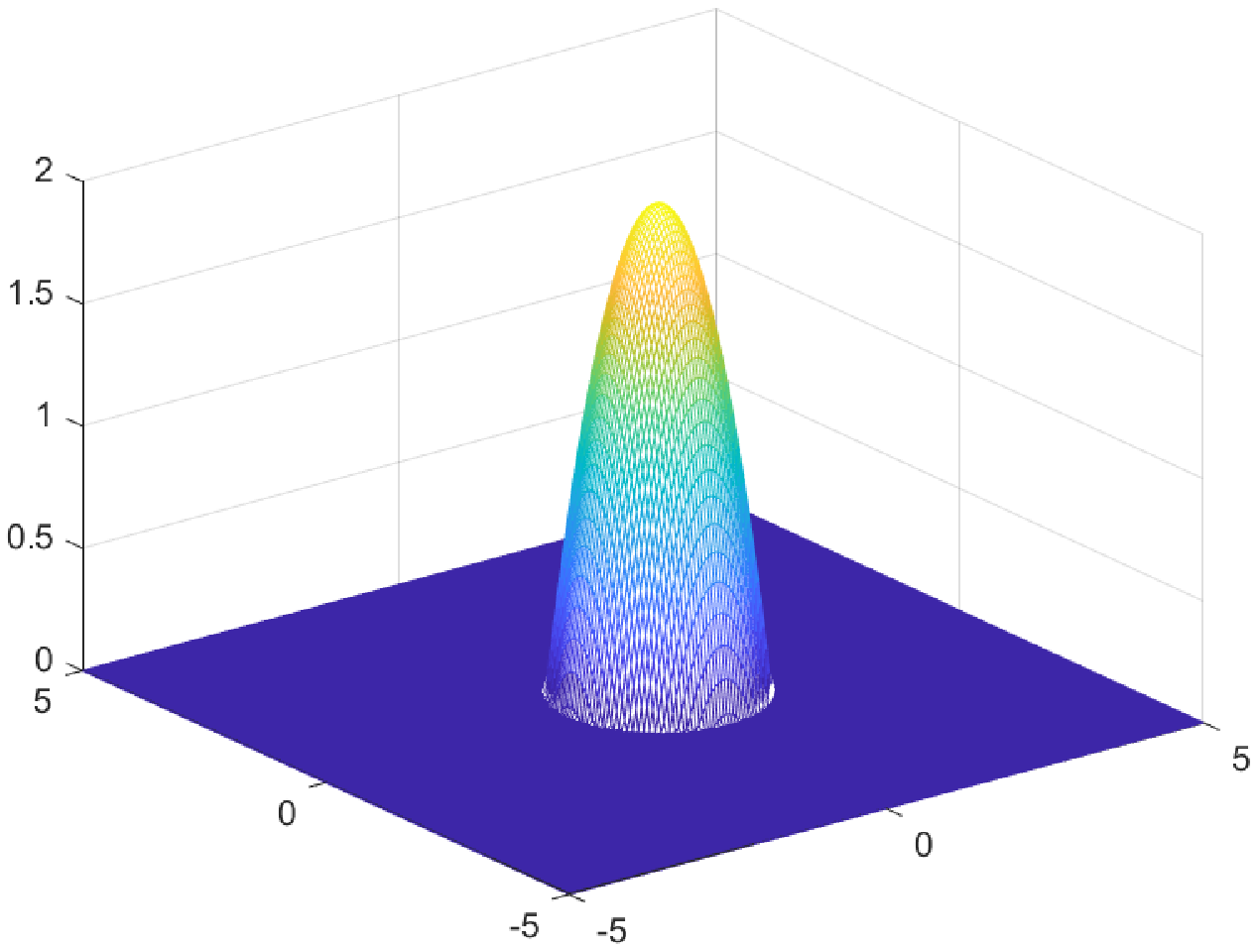}
\endminipage\hfill
\minipage{0.32\textwidth}
  \includegraphics[width=\linewidth]{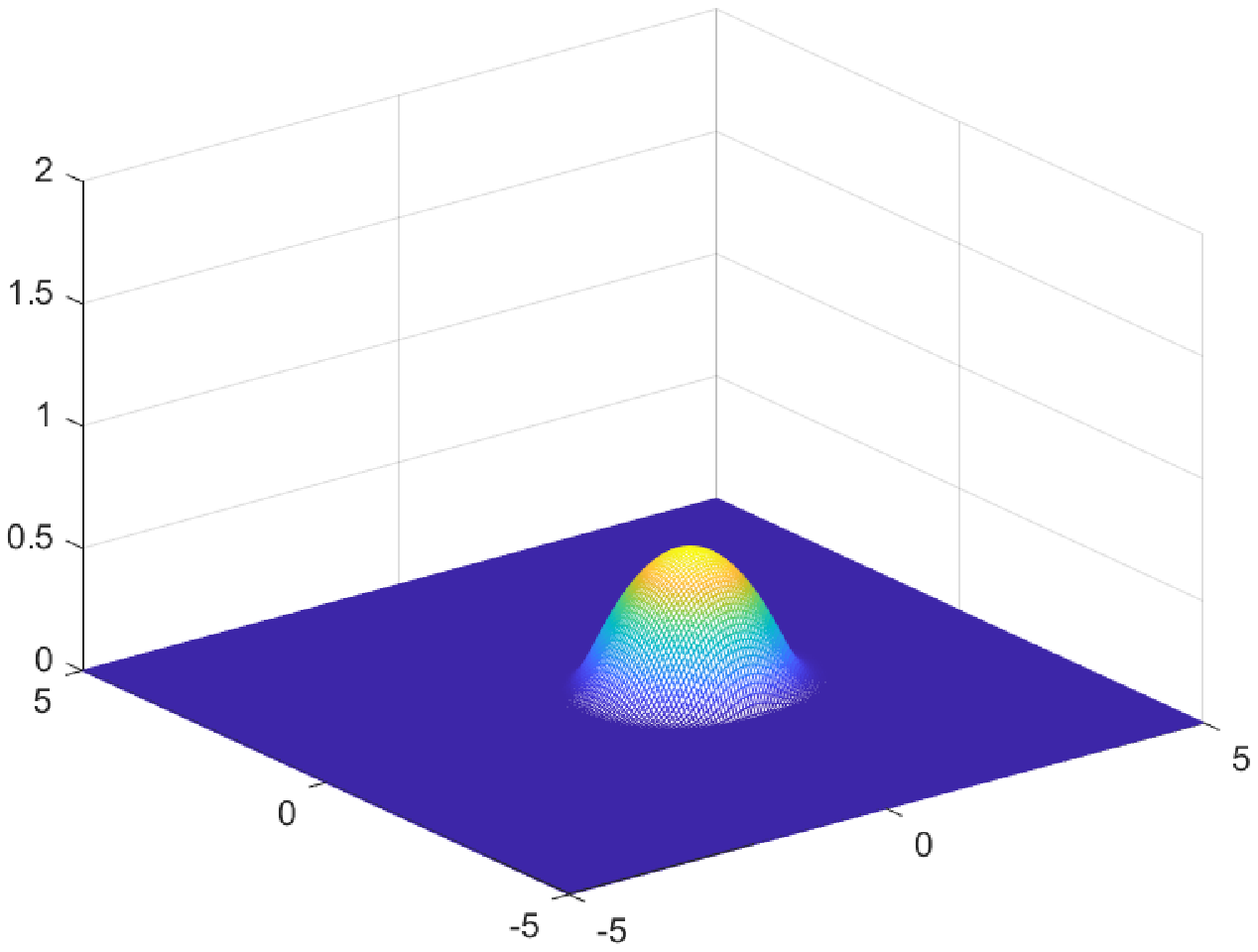}
\endminipage\hfill
\minipage{0.32\textwidth}
  \includegraphics[width=\linewidth]{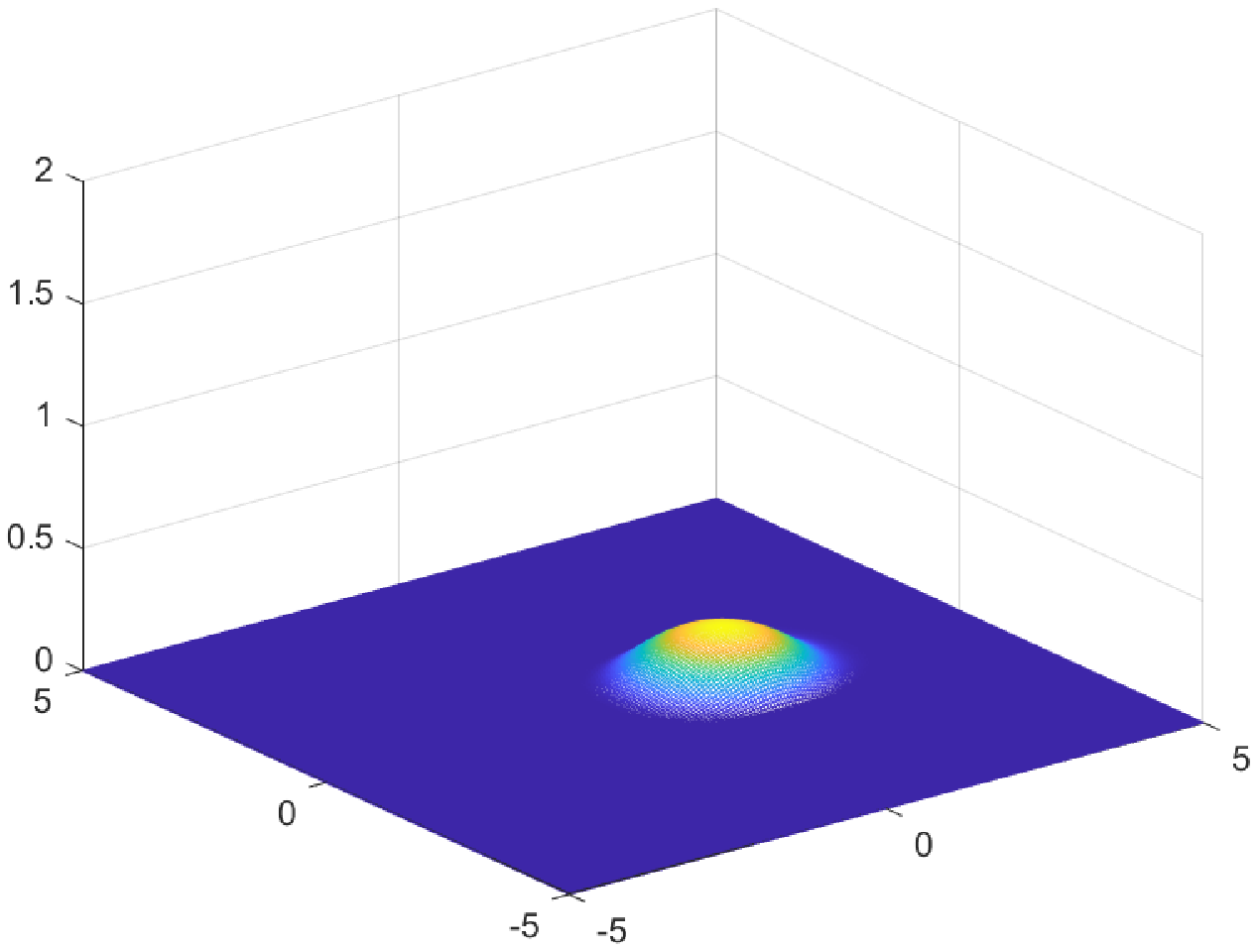}
\endminipage
\caption{Test 1: Optimal trajectory for $t=0$ (left),  $t=1$ (central) and $t=2$ (right) for HO-POD and $n=161$.}	
\label{figtransport}
\end{figure}
Finally, Table \ref{table_offline} shows the different CPU times for the offline phase for the two methods. First of all, we notice that HO-POD is faster in all cases, but we obtain a particular speed-up in presence of a one-direction convection, since the construction of the basis operates separately in each direction.

\begin{table}[bht]
\centering
\begin{tabular}{l|cc}     
 $(c_1,c_2,\sigma)$ & POD & HO-POD  \\ \hline
 $(1,0,0)$ & $17.78s$ & 0.67s \\
  $(1,1,0)$ & 18s & 1.77s \\
   $(0,0,1)$ & 16.9s & 0.84s \\
   $(1,1,1)$ & 34s & 1.73s \\
 \end{tabular}
 \caption{Offline CPU times for POD and HO-POD varying the coefficient $(c_1,c_2,\sigma)$ with $n=601$ and $\widehat{\Delta t} = 0.05$.  \label{table_offline}}
\end{table}

\subsection{Test 2: Allen-Cahn equation}

We consider the following nonlinear PDE with homogeneous Neumann boundary conditions:
\begin{equation}
\begin{cases}
\partial_s y= \sigma \Delta y +y\left(1-y^2 \right)+  y_0(x) u(s) & (x,s) \in \Omega \times [0,T],
\\
\partial_n y(x,s) =0 & (x,s) \in \partial \Omega \times [0,T], \\
y(x,0)=y_0(x) & x \in [-1,1]^2.
\end{cases}
\label{pde1}
\end{equation}
Our aim is to steer the solution to the unstable equilibrium $\overline{y} \equiv 0$ minimizing the following cost functional
\begin{equation*}
J_{y_0,t}(u) = \int_t^T \left(  \int_{\Omega} |y(x,s)|^2 dx +   \gamma |u(s)|^2 \right) ds + \int_{\Omega} |y(x,T)|^2  dx.
\end{equation*}
We fix $T=1$, $\gamma=0.01$, $\sigma=0.1$, $\Omega = [-1,1]^2$ $U=[-2,0]$ and $y_0(x)=2+\cos(2 \pi x_1) \cos (2 \pi x_2)$. Furthermore, we set $\Delta t=0.1$, $\tau=10^{-3}$ and we discretize the domain $[-1,1]^2$ with $601$ equidistant points, obtaining a grid of $361201$ points. In the offline phase we consider $2$ discrete controls and we construct a rough tree with $2047$ nodes. Since the problem is not linear, we apply the HO-POD-DEIM strategy and the dimensions of the basis turns to be $k_1=k_2=p_1=p_2=5$, hence the low dimension solution lives in $\mathbb{R}^{5 \times 5}$. The offline stage took $174$ seconds. In this example the dynamical system is nonlinear and we do not have any a priori estimate for the introduced pruning criteria. Hence, we are going to apply the statistical pruning discussed in Section \ref{stat_pru}. We fix the ratio $\rho=0.2$ and we iterative the statistical pruning strategy explained in Algorithm \ref{alg_1} with a stopping tolerance $tol=10^{-4}$.
In the left panel of Figure \ref{fig_test2_card} we show the cardinality of the low dimensional tree in logarithm scale. We recall that the cardinality of the full tree would be $O(M^{11})$, where $M$ is the number of discrete controls, reaching an order of $\approx$ $10^{17}$ in the case of $33$ controls. The application of the statistical pruning achieves a great improvement in terms of memory storage, gaining almost 12 orders of magnitudes with respect to the full tree with $33$ discrete controls.
The total cost varying the number of controls is displayed in the left panel of Figure \ref{fig_test2_card}. The cost functional shows a decreasing behaviour as expected and the algorithms stops with $33$ controls since the stopping rule has been satisfied. In Figures \ref{fig_test2_traj1}-\ref{fig_test2_traj2} the optimal trajectories at time instances $t\in \{0,0.5,1\}$ and the control signal are displayed. We note that the control signal is driving the solution to the unstable equilibrium $\tilde{y} \equiv 0$.


\begin{figure}[htb!]
\minipage{0.49\textwidth}
  \includegraphics[width=\linewidth]{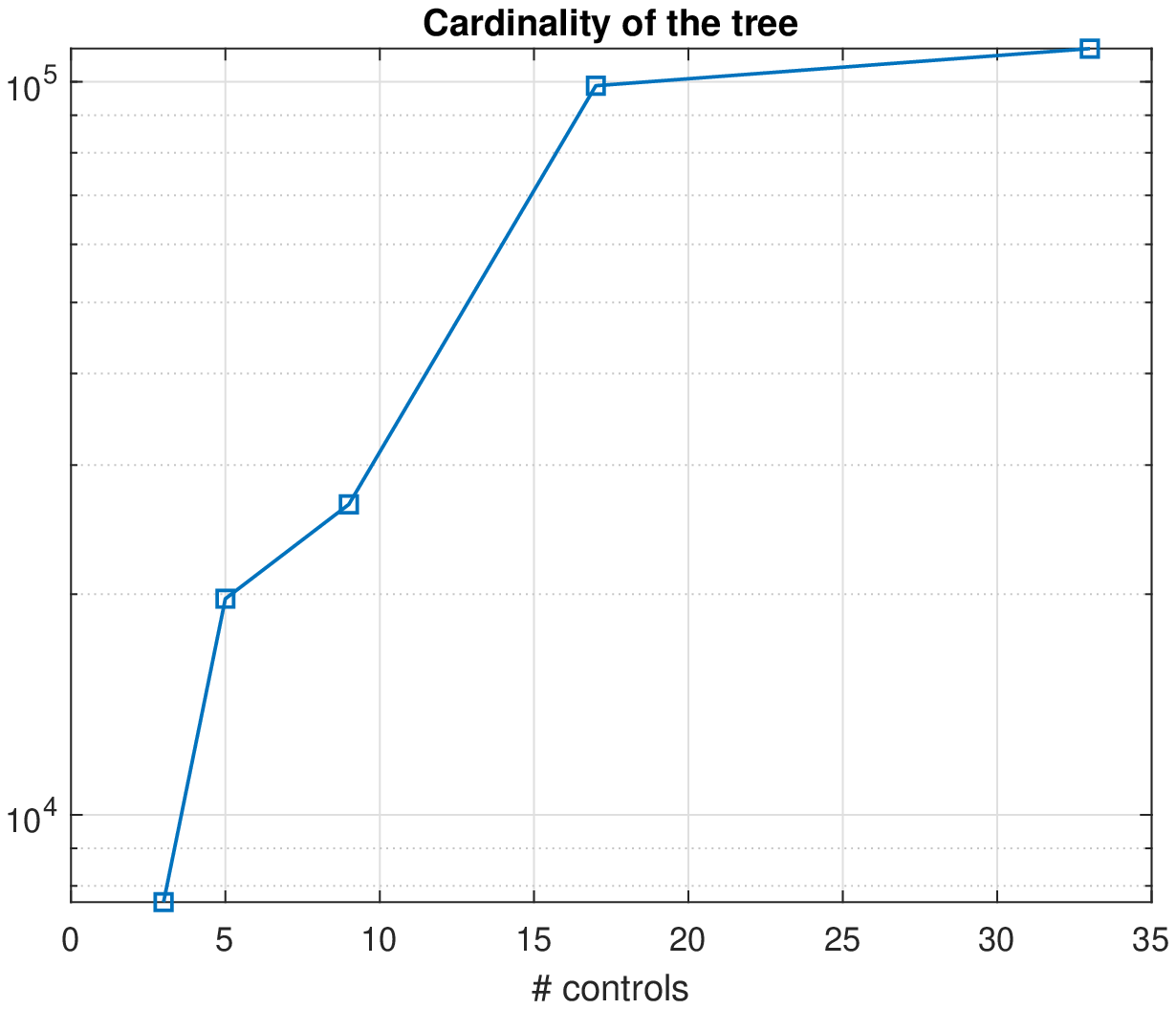}
\endminipage\hfill
\minipage{0.49\textwidth}
  \includegraphics[width=\linewidth]{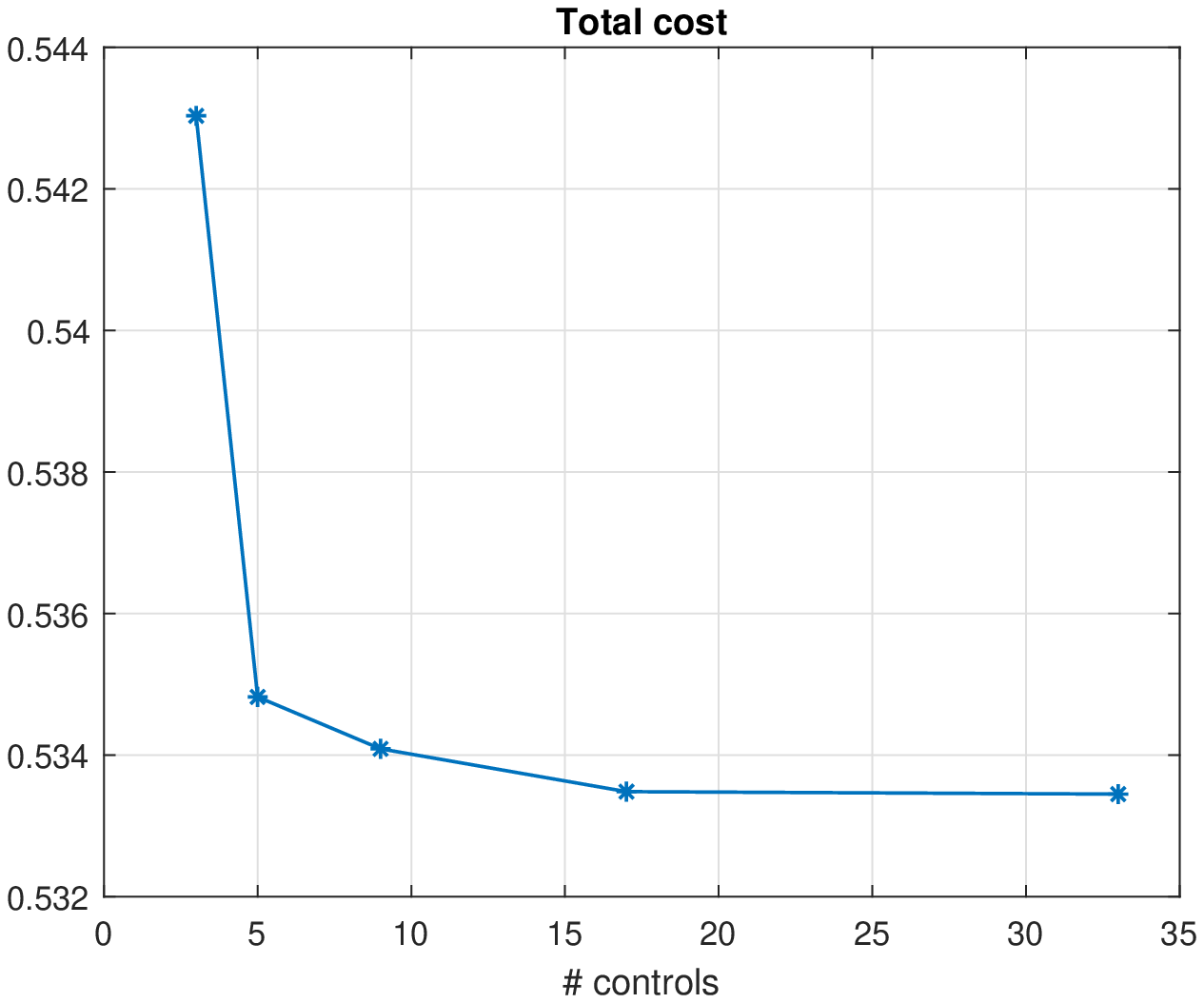}
\endminipage
 \caption{Test 2: Cardinality of the tree in logarithmic scale (left) and total cost (right) for $\Delta t=0.1$ and varying the number of controls.}	
 \label{fig_test2_card}
 \end{figure}

\begin{figure}[htb!]
\minipage{0.49\textwidth}
  \includegraphics[width=\linewidth]{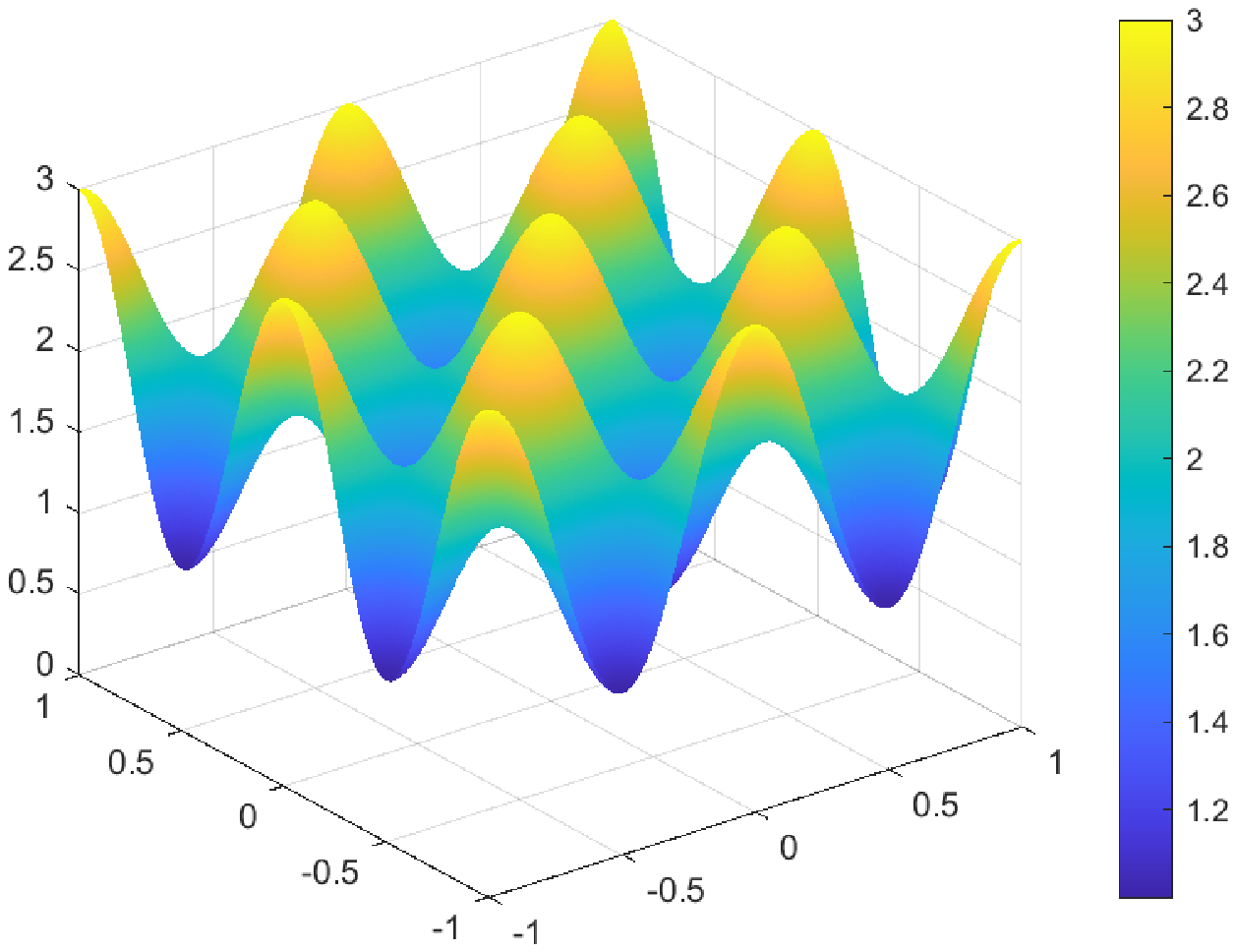}
\endminipage\hfill
\minipage{0.49\textwidth}
  \includegraphics[width=\linewidth]{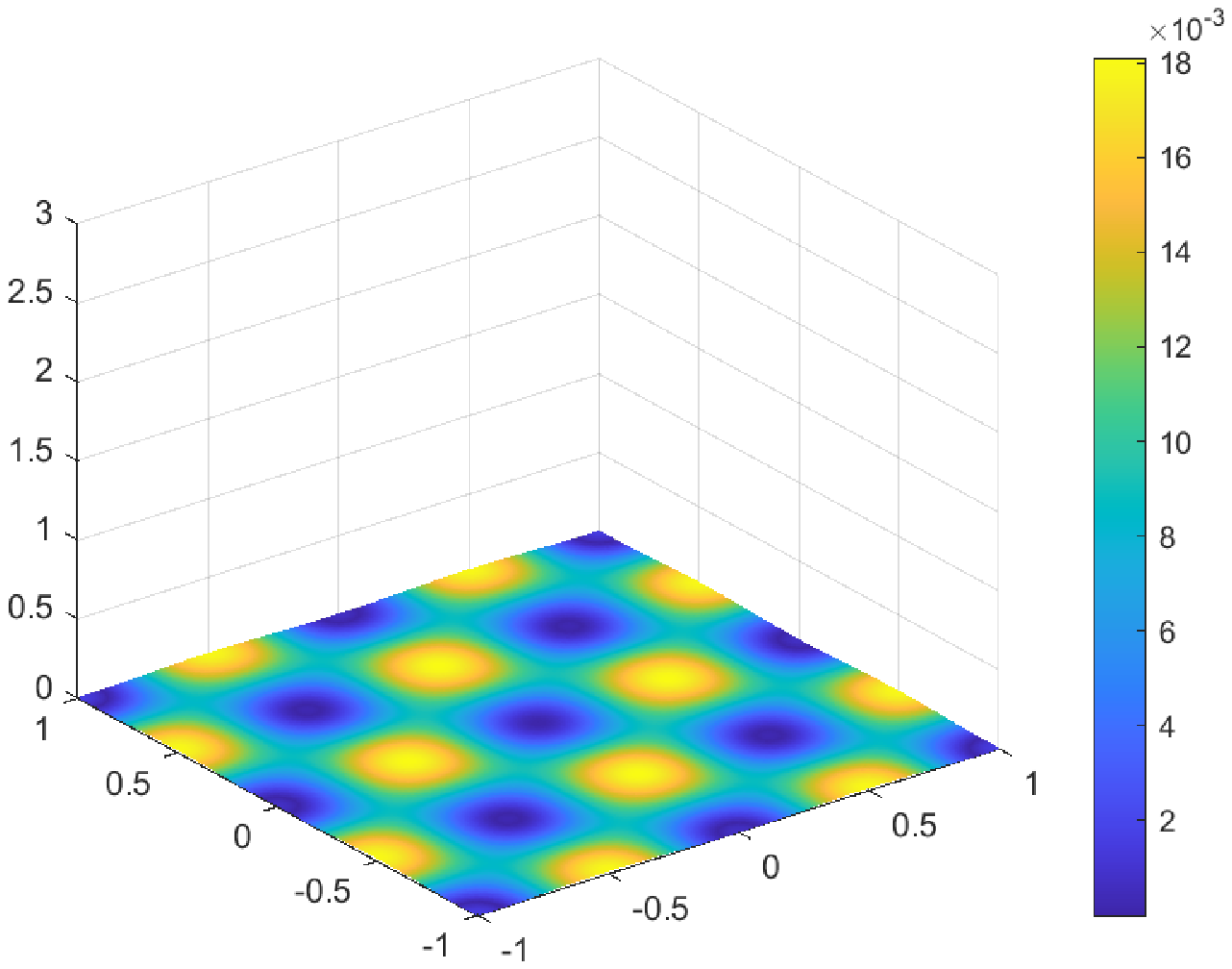}
\endminipage
 \caption{Test 2: Optimal trajectory at $t=0$ (left) and $t=0.5$ (right) for $\Delta t=0.1$ and $M=33$.}	
 \label{fig_test2_traj1}
 \end{figure}
 
 \begin{figure}[htb!]
\minipage{0.49\textwidth}
  \includegraphics[width=\linewidth]{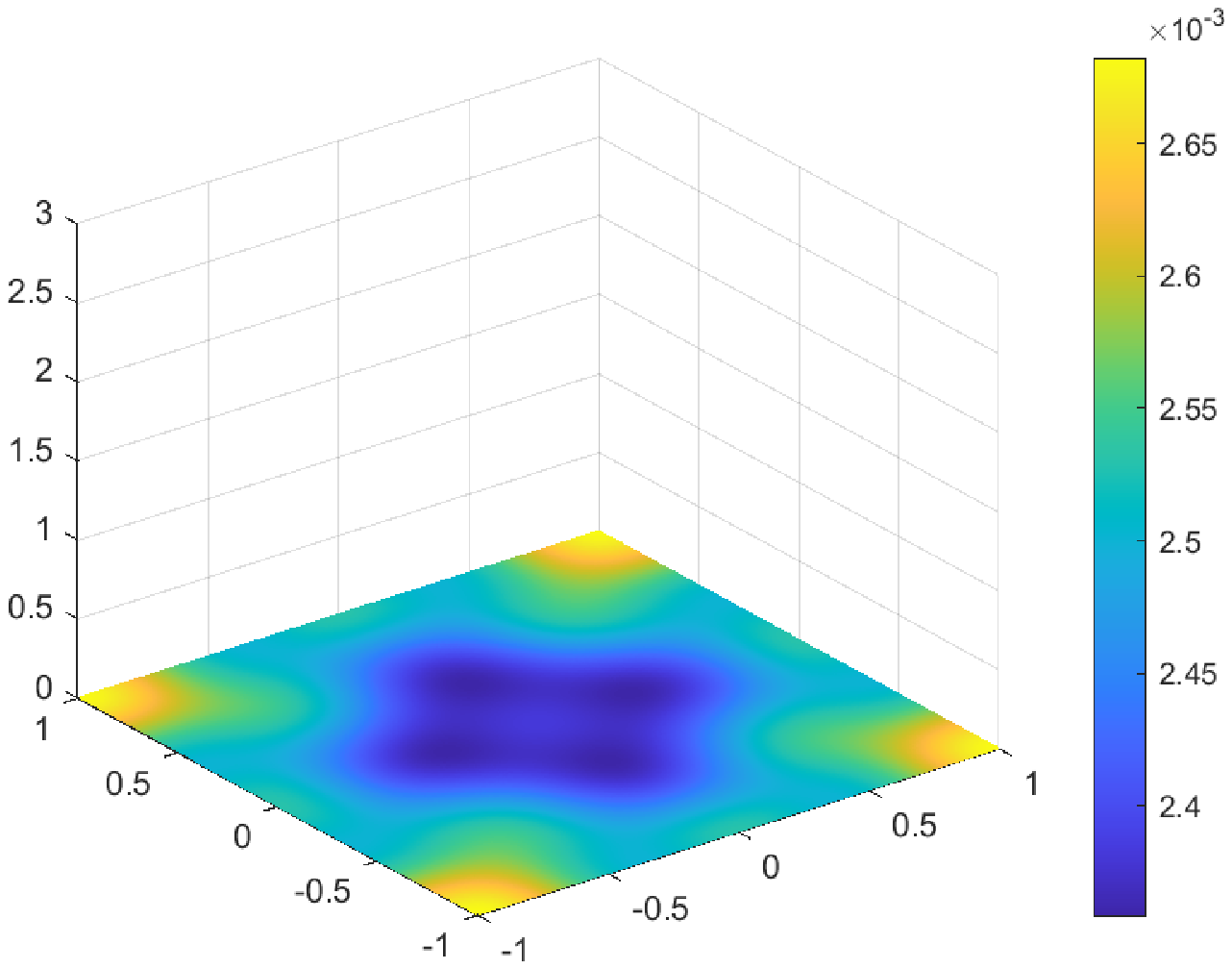}
\endminipage\hfill
\minipage{0.49\textwidth}
  \includegraphics[width=\linewidth]{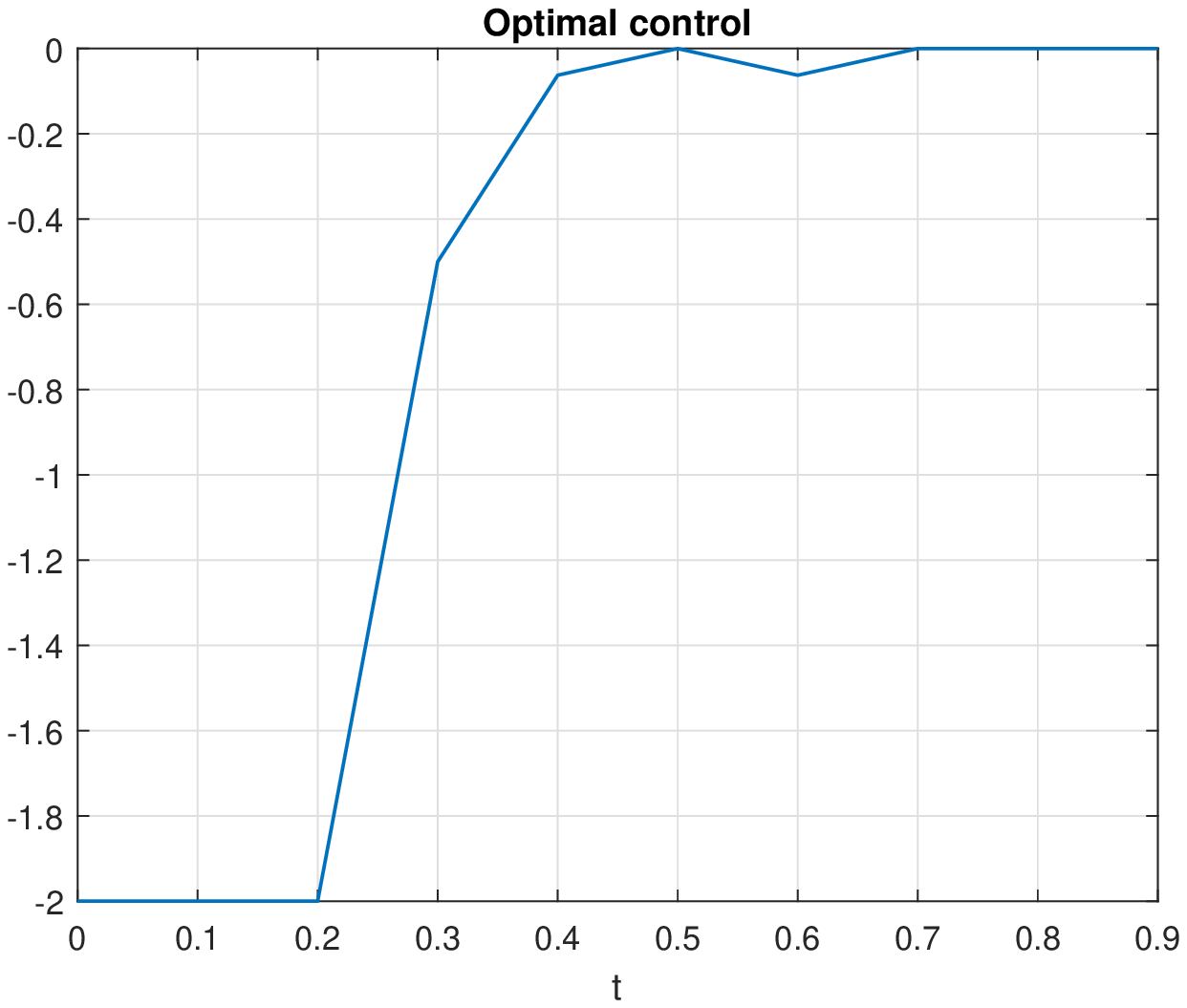}
\endminipage
 \caption{Test 2: Optimal trajectory at $t=1$ (left) and control signal (right) for $\Delta t=0.1$ and $M=33$.}
 \label{fig_test2_traj2}
 \end{figure}

\subsection{Test 3: 3D viscous Burgers' equation}

Here we consider the nonlinear 3D viscous Burgers' equation (see, e.g., \cite{GAO2017}) given by
\begin{equation}
\begin{cases}
\label{hamburger3d}
\partial_t y_1 &=  \frac{1}{r}\Delta y_1 -  \underline{y} \cdot \nabla y_1+y_1 u\\
\partial_t y_2 &=  \frac{1}{r}\Delta y_2 - \underline{y} \cdot \nabla y_2 + y_2 u,\\
\partial_t y_3 &=  \frac{1}{r}\Delta y_3 - \underline{y} \cdot \nabla y_3 + y_3 u,\\
\end{cases}
\end{equation}
where $y_1(x_1,x_2,x_3,t)$, $y_2(x_1,x_2,x_3,t)$ and $y_3(x_1,x_2,x_3,t)$ are the three velocities to be determined, with $x=(x_1,x_2,x_3) \in [0,1]^3$, $t \in [0,1]$ and the Reynold's number $r = 100$. Furthermore, the system is subject to homogeneous Dirichlet boundary conditions and initial states
\begin{equation*}
\begin{split}
y_1(x, y, z, 0) &= \frac{1}{10}\sin(2\pi x_1)\sin(2\pi x_2)\cos(2\pi x_3) \\
y_2(x, y, z, 0) &=  \frac{1}{10}\sin(2\pi x_1)\cos(2\pi x_2)\sin(2\pi x_3)\\
y_3(x, y, z, 0) &=  \frac{1}{10}\cos(2\pi x_1)\sin(2\pi x_2)\sin(2\pi x_3). \\
\end{split}
\end{equation*}
A finite difference space discretization in the cube yields a system of ODEs of the form \eqref{arraybigsystem}, with nonlinear functions given by
\begin{equation*}
\begin{split}
{\cal F}_i({\cal D}_i({\pmb{\cal Y}}_i),{\pmb{\cal Y}}_1, {\pmb{\cal Y}}_2,{\pmb{\cal Y}}_3, t) &= \sum_{k=1}^3 ({\pmb{\cal Y}}_i \times_k {B}_{ki}) \circ {\pmb{\cal Y}}_k,
\end{split}
\end{equation*}
for $i = 1,2,3$, where ${B}_{1i} \in \mathbb{R}^{n \times n}$, ${B}_{2i} \in \mathbb{R}^{n \times n}$ and ${B}_{3i} \in \mathbb{R}^{n \times n}$ contain the coefficients for a first order centered difference space discretization in the $x_1-$, $x_2-$ and $x_3-$ directions respectively, and $n$ is the dimension of the discretized tensor in each spatial direction. For a more detailed discussion on the space discretization and HO-POD-DEIM model reduction of {\it systems of ODEs} in array form, we point the reader to \cite{kirsten.22}, as well as the companion manuscript \cite{KSF2023}.

We consider the following cost functional
\begin{equation}\label{cost_test}
J_{y_0,t}(u) = \int_t^T \left(  \int_{\Omega}\sum_{i = 1}^3 |y_i(x,s)|^2 dx + \dfrac{1}{10}  |u(s)|^2 \right) ds + \int_{\Omega} \sum_{i = 1}^3 |y_i(x,T)|^2  dx.
\end{equation}
The control $u(t)$ will be taken in the following admissible set of controls
$$
\mathcal{U}=\{u:[0,T] \rightarrow [-2,0]\}.
$$
We therefore construct one tree for the control $u$ containing the approximate solution of each of the three equations at its nodes. Constructing and storing this tree of course leads to extremely demanding memory requirements and computational effort to construct the approximations spaces for the reduced models.

We therefore use this experiment to illustrate the massive computational gain of the HO-POD-DEIM method, in combination with the snapshot selection algorithm and the low-rank storage algorithm. We first investigate the computational load required in the offline phase by the HO-POD-DEIM method as well as standard POD-DEIM applied to the system (\ref{hamburger3d}) discretized in vector form. For the vectorized system we also apply a semi-implicit Euler time discretization to each of the three equations, and each linear system is solved using the Matlab function {\tt pcg} preconditioned with an incomplete Cholesky factorization with a drop tolerance of $10^{-4}$. 

Below we illustrate the computational load both in terms of CPU time and memory requirements. On the left of Figure \ref{figten} we plot the computational time required by both methods to construct the full-dimensional tree, with $N_t = 10$ and two controls, whose nodal values are used to construct either the HO-POD-DEIM basis or the standard POD-DEIM basis. To construct the tree, all nodal values from the previous time level need to be stored. To this end, one of the computational bottlenecks in the construction of the reduced model is memory requirements. Consequently, we further illustrate the power of the proposed algorithm on the right of Figure \ref{figten}, where we plot the maximum memory required  (in {\tt mb}) at any point in the offline phase of the respective algorithms. The plot indicates a massive difference in memory requirements, mainly related to the low-rank basis construction used in the HO-POD-DEIM algorithm, as well as the low-rank memory allocation method discussed in Section \ref{mem}. Furthermore we notice, that no data points are plotted in the vector case for $n > 60$, as this is where the computer ran out of its available computational memory. Both plots are with respect to increasing $n$ and we select $\tau = 10^{-2}$ and $\kappa = 20$ a priori.

\begin{figure}[htb!]
\minipage{0.49\textwidth}
 \includegraphics[width=\linewidth]{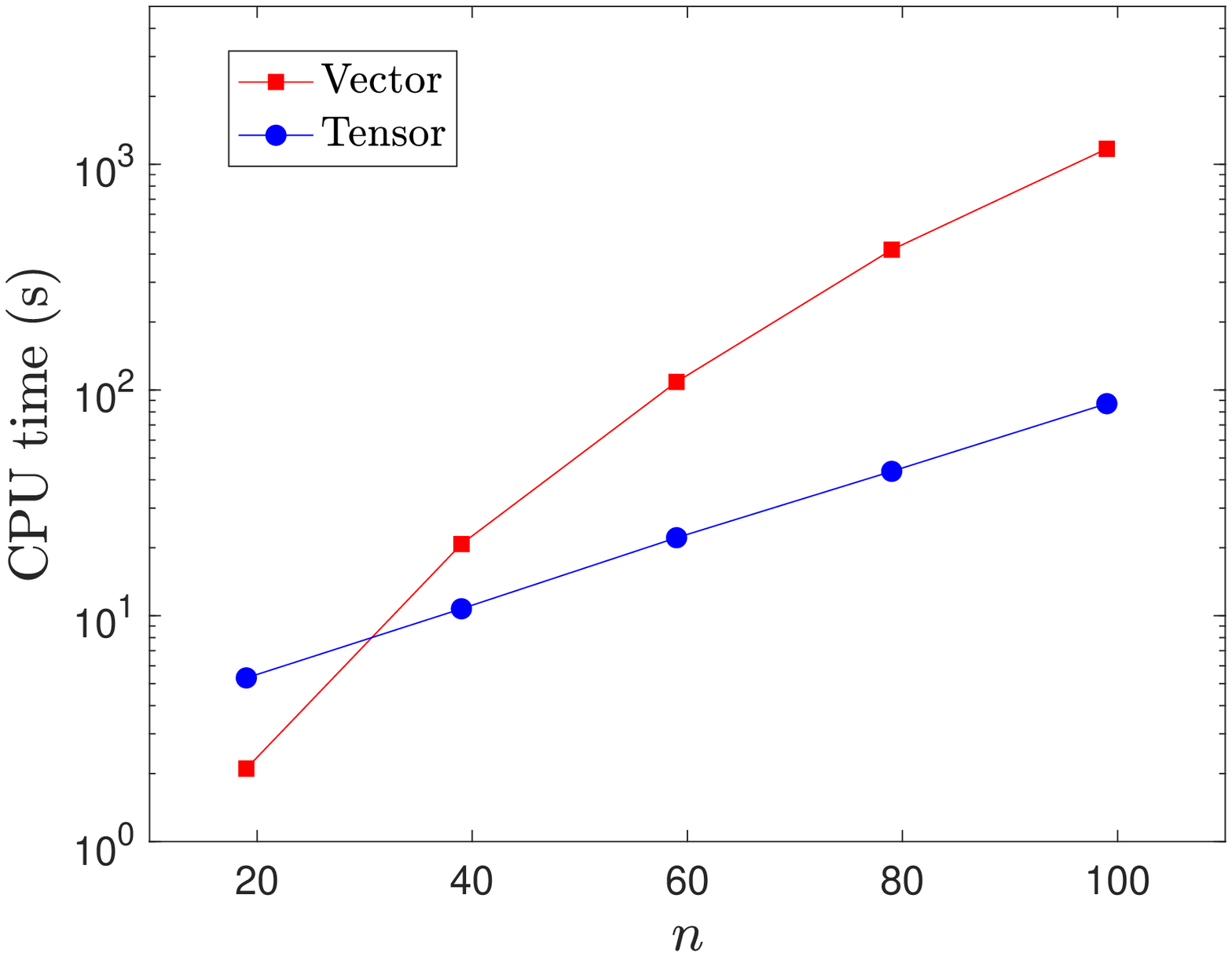}
\endminipage\hfill
\minipage{0.49\textwidth}
 \includegraphics[width=\linewidth]{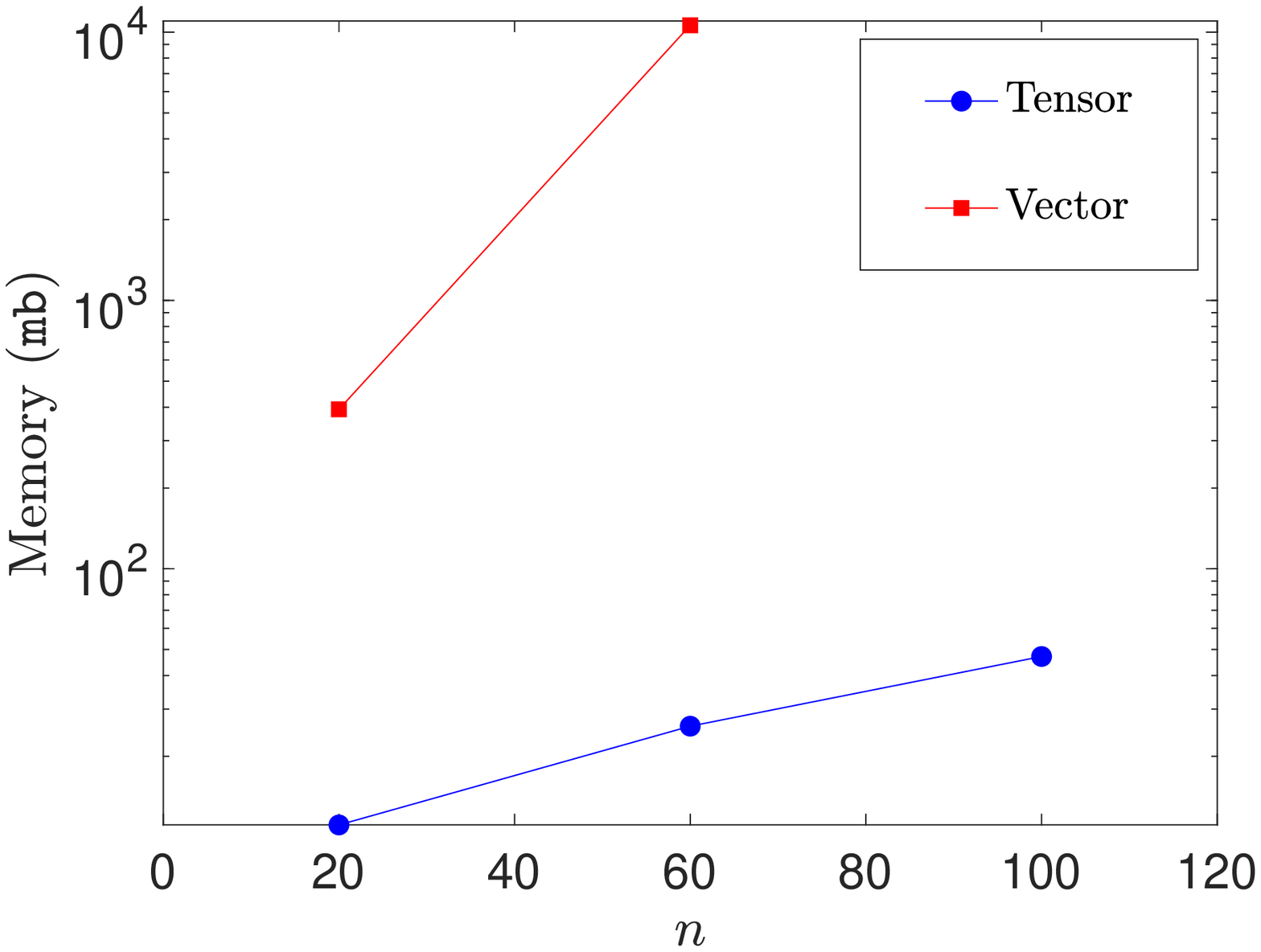}
\endminipage
    \caption{Test: CPU time (left) and memory requirements (right) for both methods applied to (\ref{hamburger3d}). \label{figten}}
\end{figure}

In what follows we consider the reduced model constructed by the HO-POD-DEIM method and investigate the efficiency of the reduction. In Table \ref{comp5} we indicate the dimensions of the reduced approximation spaces determined to comply with $\tau = 10^{-2}$. To construct a tree with ten time steps and two controls, pruned by the standard geometric pruning technique, the reduced model with dimensions as in Table \ref{comp5} required merely 19 seconds in comparison to the 361 seconds required by the full order model.

\begin{table}[bht]
\centering
\begin{tabular}{l|ccccccc}
      
      $y$  &   $k_1$       & $k_2$       &$k_3$    &   $p_1$        &  $p_2$     &$p_3$      & {\sc error} \\ \hline
$y_1$  & 6       & 11   & 12        & 10& 18& 19        & $2\cdot 10^{-2}$   \\ 
$y_2$  & 6         & 15 & 13        & 6 &20&17        & $3\cdot 10^{-2}$    \\ 
$y_3$  & 6          & 12  & 13        & 5 &18&18        & $2\cdot 10^{-2}$    \\ \hline
\end{tabular}
\caption{Dim. of basis and the average relative error compared to the full order model with dimension $n = 60$ and $\tau = 10^{-2}$.  \label{comp5}}
\end{table}

Finally, we also plot the cost functional in Figure \ref{figcost} for both the full and reduced order models as well as the optimal trajectories (unfolded along the first mode) for all three equations at $t = 0$, $t = 0.5$ and $t = 1$ in Figure \ref{fig:images}. Both these plots, indicate the convergence to the equilibrium of the reduced model. We note that there is a visual superposition of the two curves of the cost functional, demonstrating the effectiveness of the proposed methodology for determining the optimal trajectory.

\begin{figure}[htbp]		
\centering
	\includegraphics[scale=0.3]{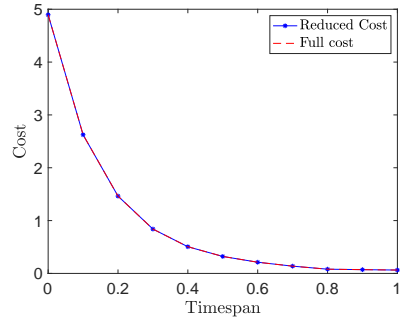}	
\caption{Cost functional for the optimal control.}	
\label{figcost}
\end{figure}

\begin{figure}[htb]
    \centering 
\begin{subfigure}{0.25\textwidth}
 \includegraphics[width=\linewidth]{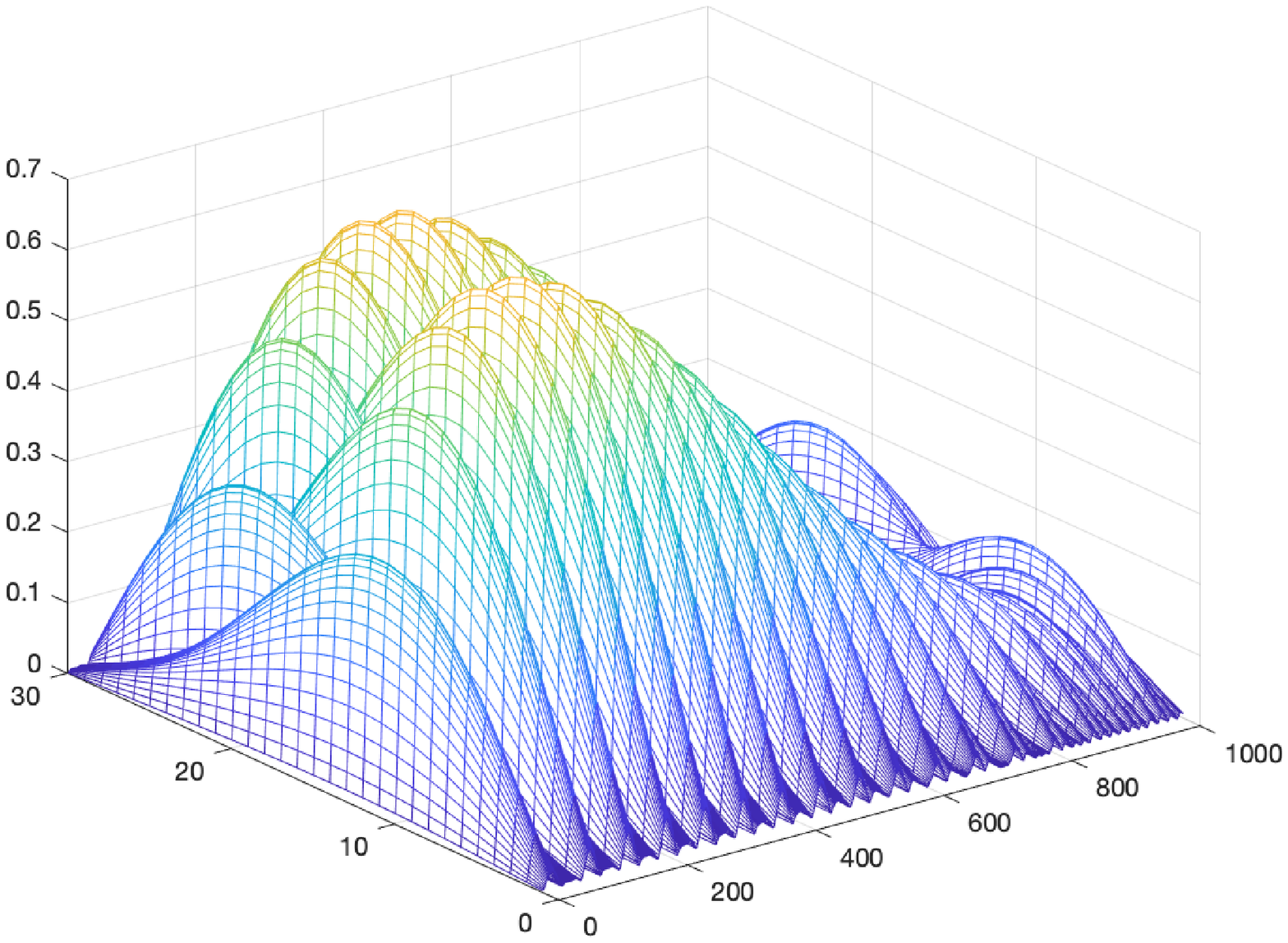}
 \label{fig:1}
\end{subfigure}\hfil 
\begin{subfigure}{0.25\textwidth}
 \includegraphics[width=\linewidth]{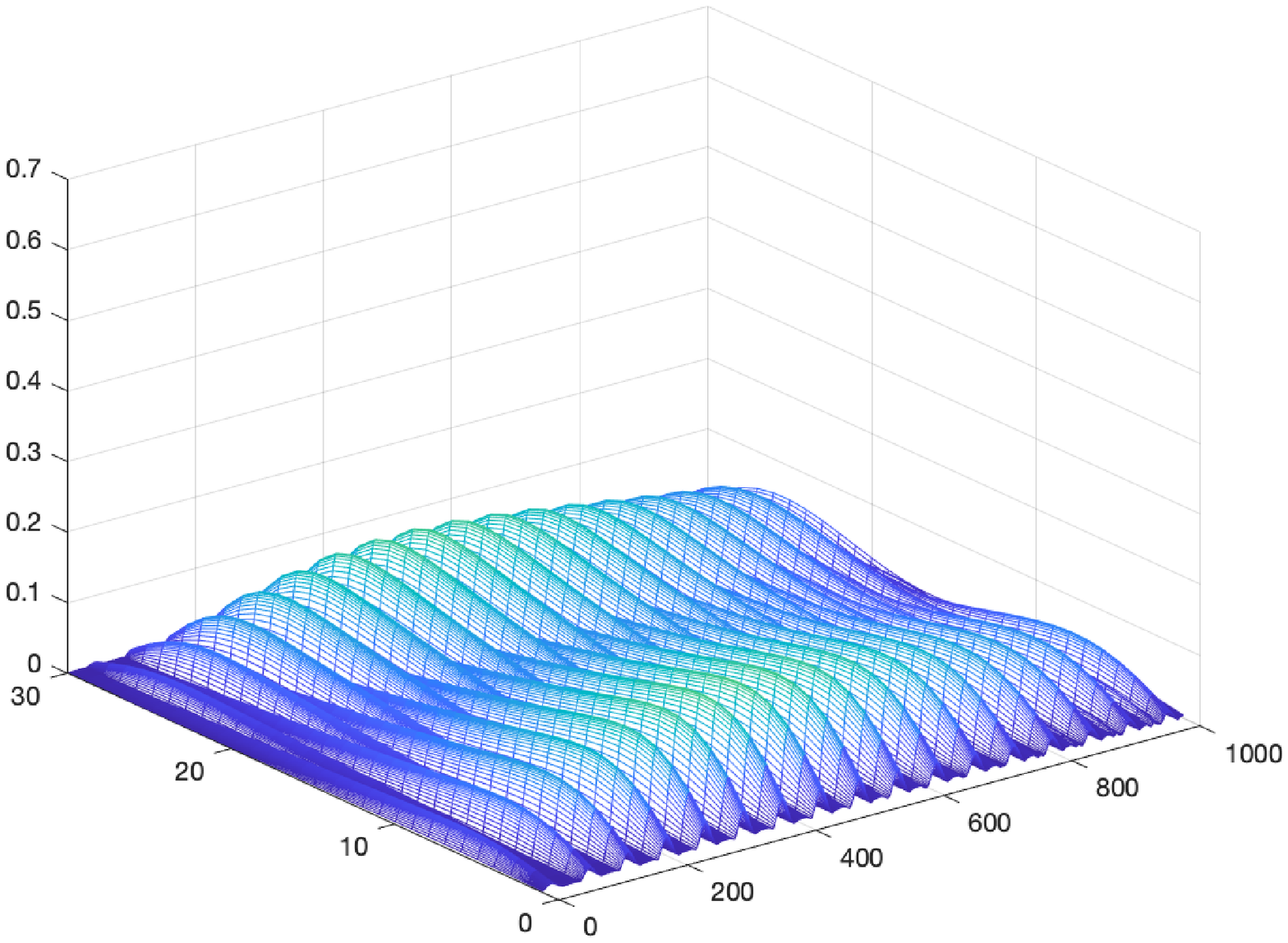}
 \label{fig:2}
\end{subfigure}\hfil 
\begin{subfigure}{0.25\textwidth}
 \includegraphics[width=\linewidth]{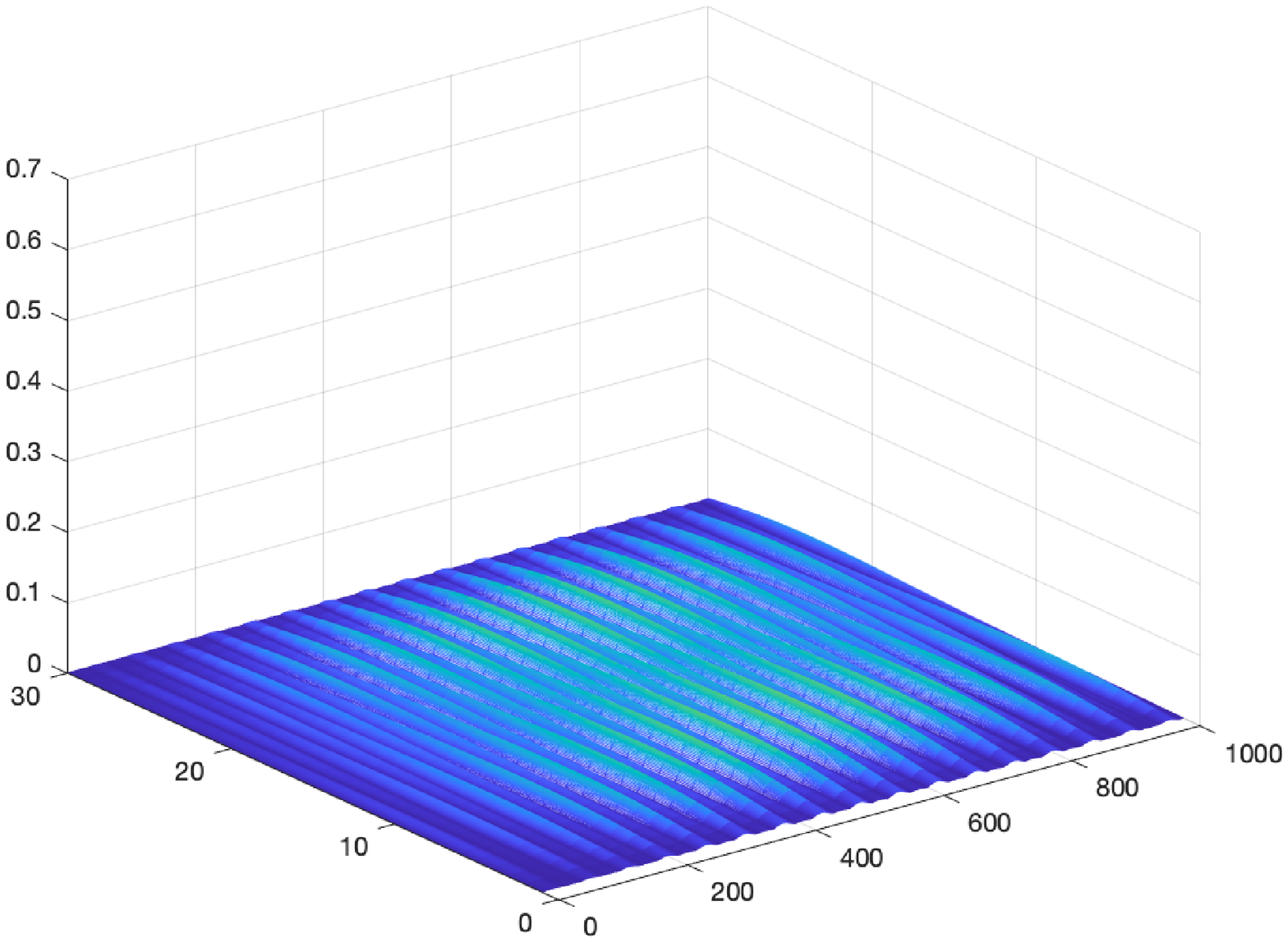}
 \label{fig:3}
\end{subfigure}

\medskip
\begin{subfigure}{0.25\textwidth}
 \includegraphics[width=\linewidth]{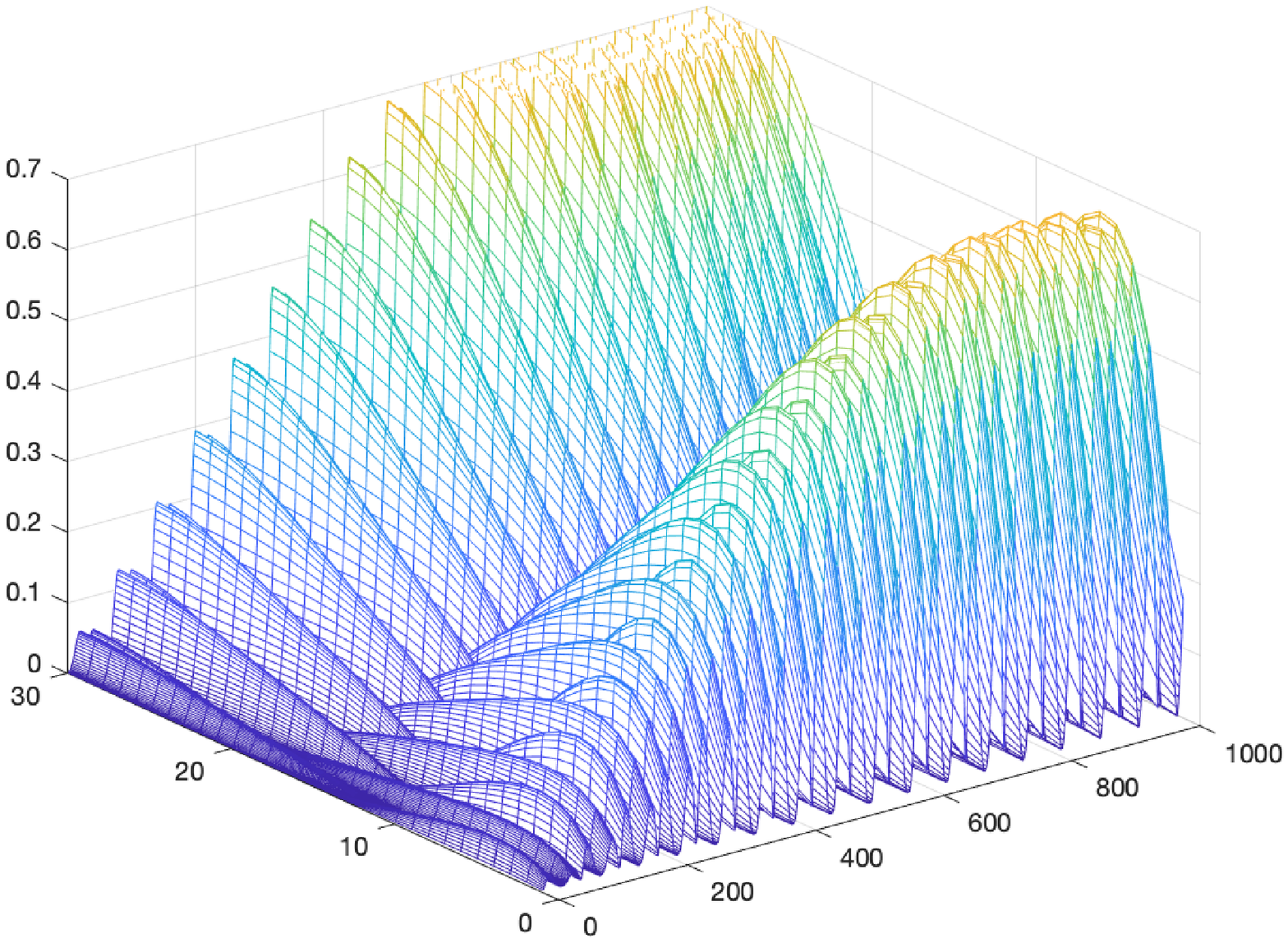}
 \label{fig:4}
\end{subfigure}\hfil 
\begin{subfigure}{0.25\textwidth}
 \includegraphics[width=\linewidth]{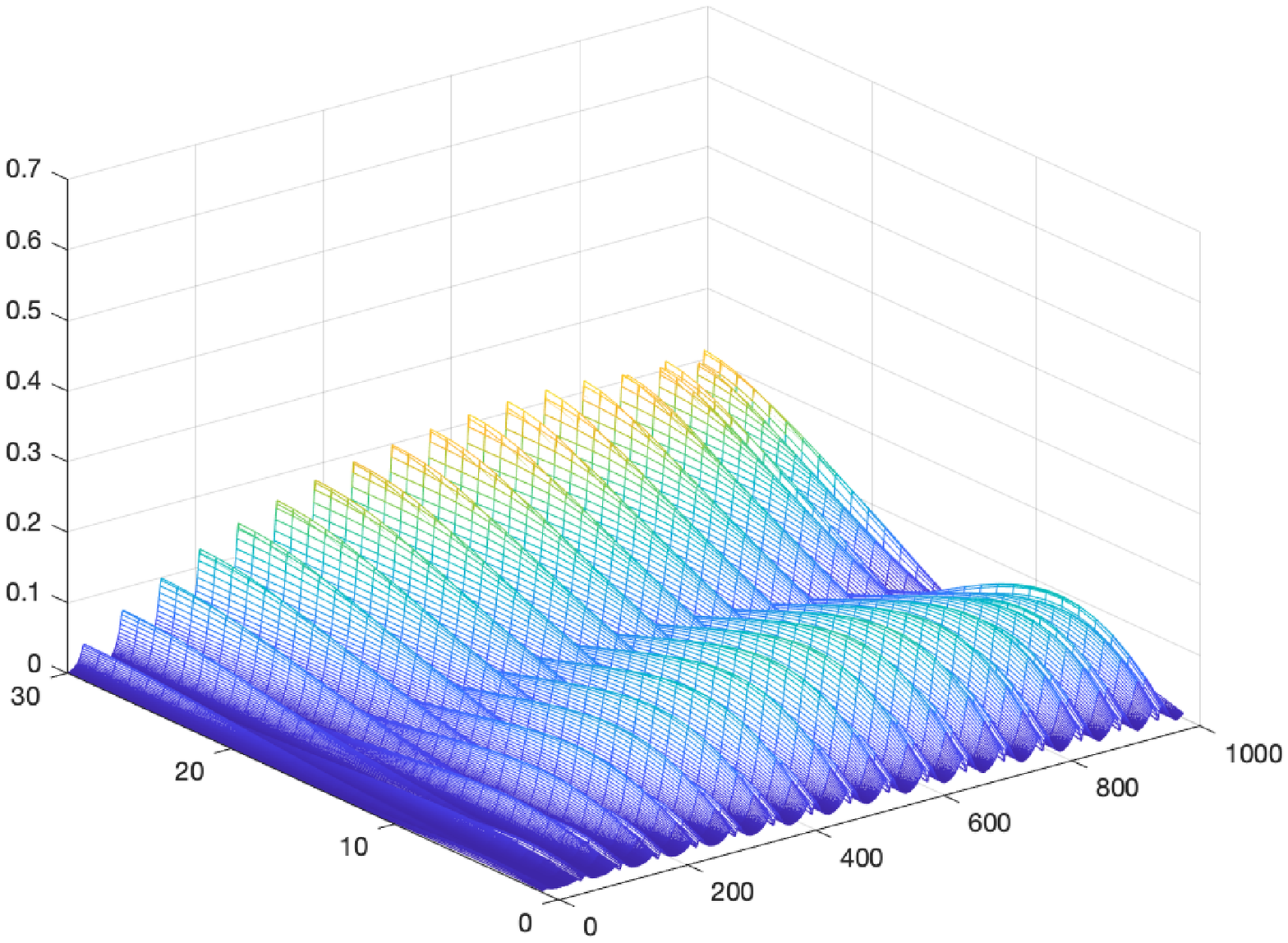}
 \label{fig:5}
\end{subfigure}\hfil 
\begin{subfigure}{0.25\textwidth}
 \includegraphics[width=\linewidth]{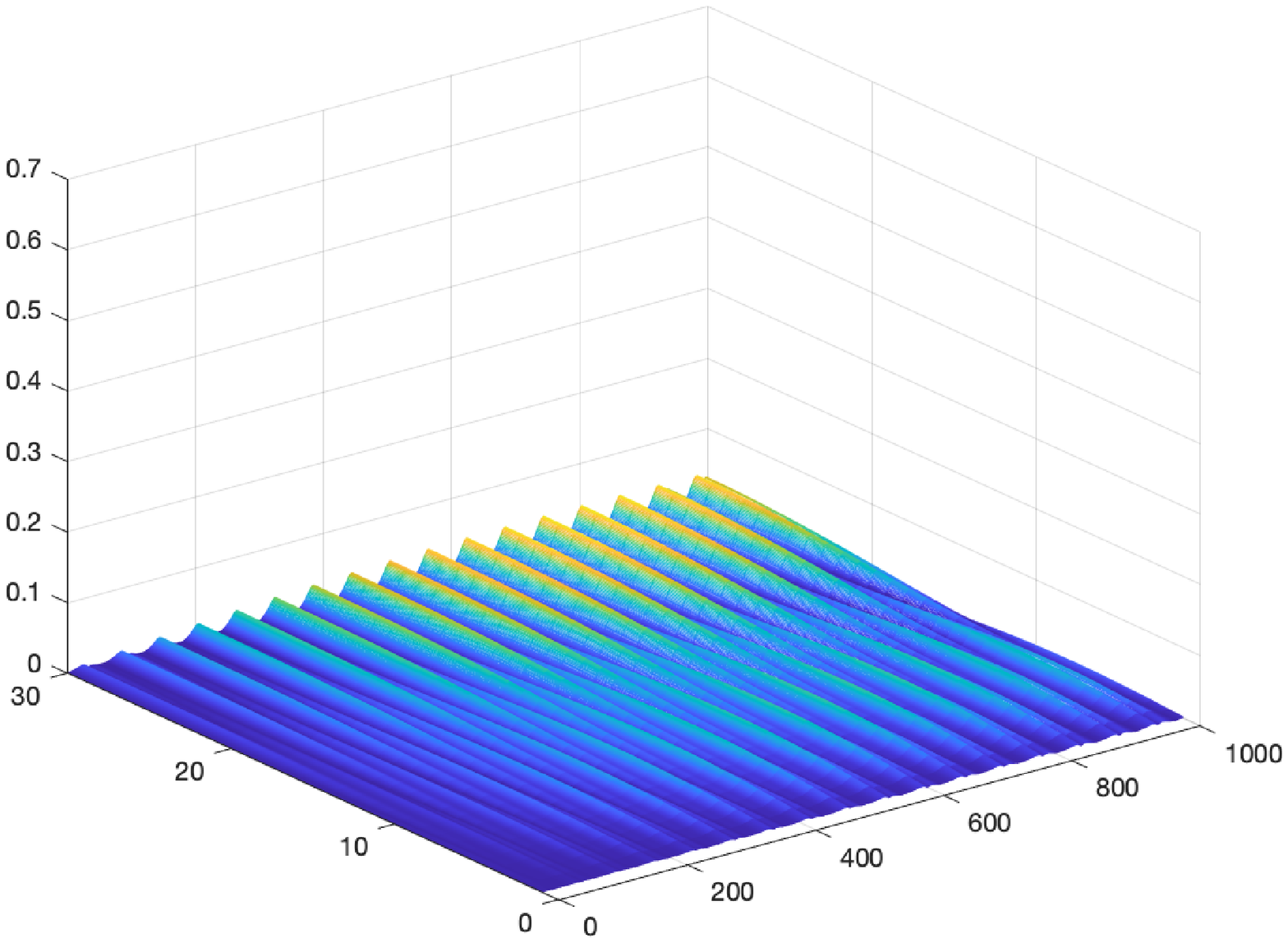}
 \label{fig:6}
\end{subfigure}

\medskip
\begin{subfigure}{0.25\textwidth}
 \includegraphics[width=\linewidth]{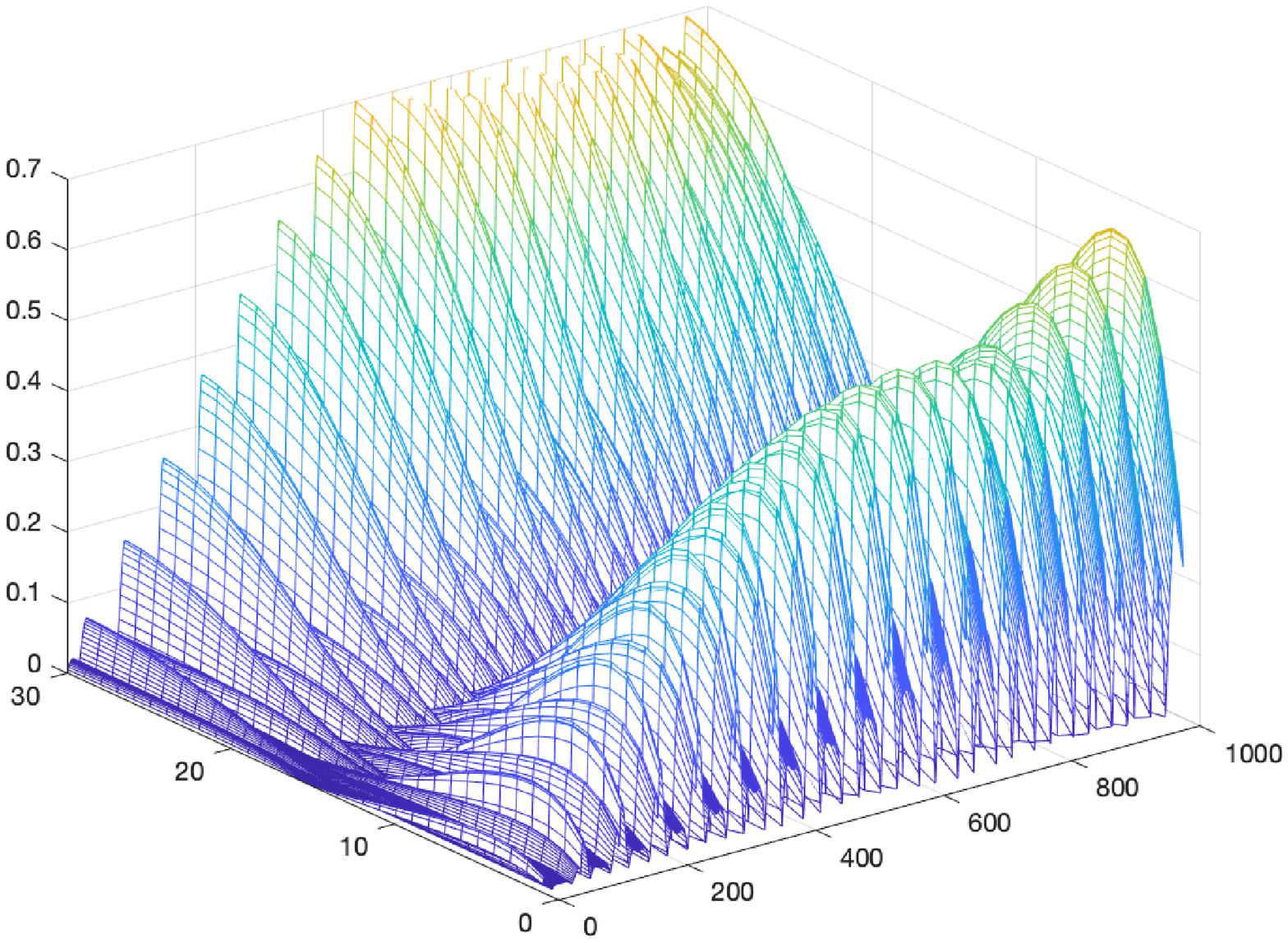}
 \label{fig:7}
\end{subfigure}\hfil 
\begin{subfigure}{0.25\textwidth}
 \includegraphics[width=\linewidth]{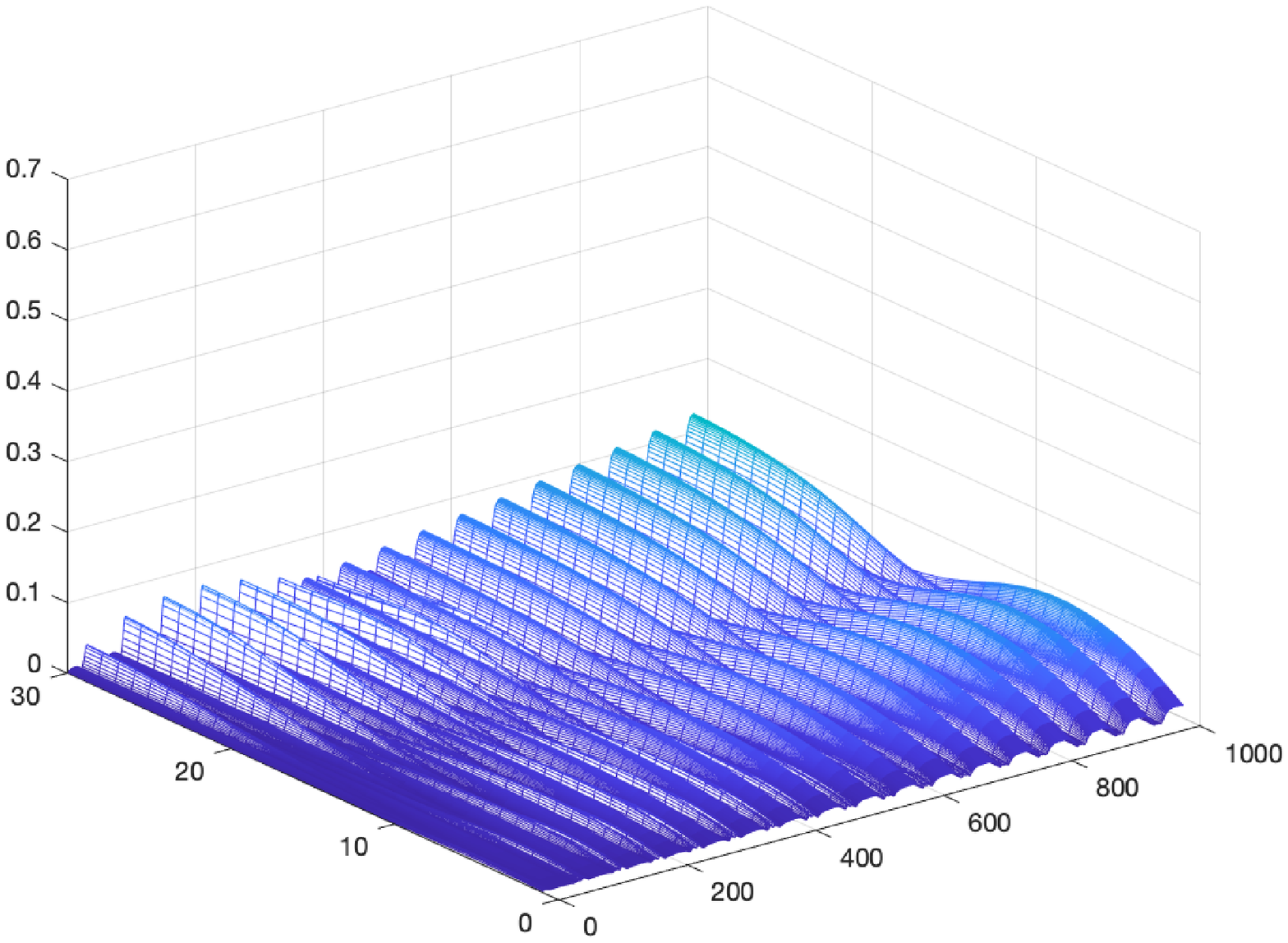}
 \label{fig:8}
\end{subfigure}\hfil 
\begin{subfigure}{0.25\textwidth}
 \includegraphics[width=\linewidth]{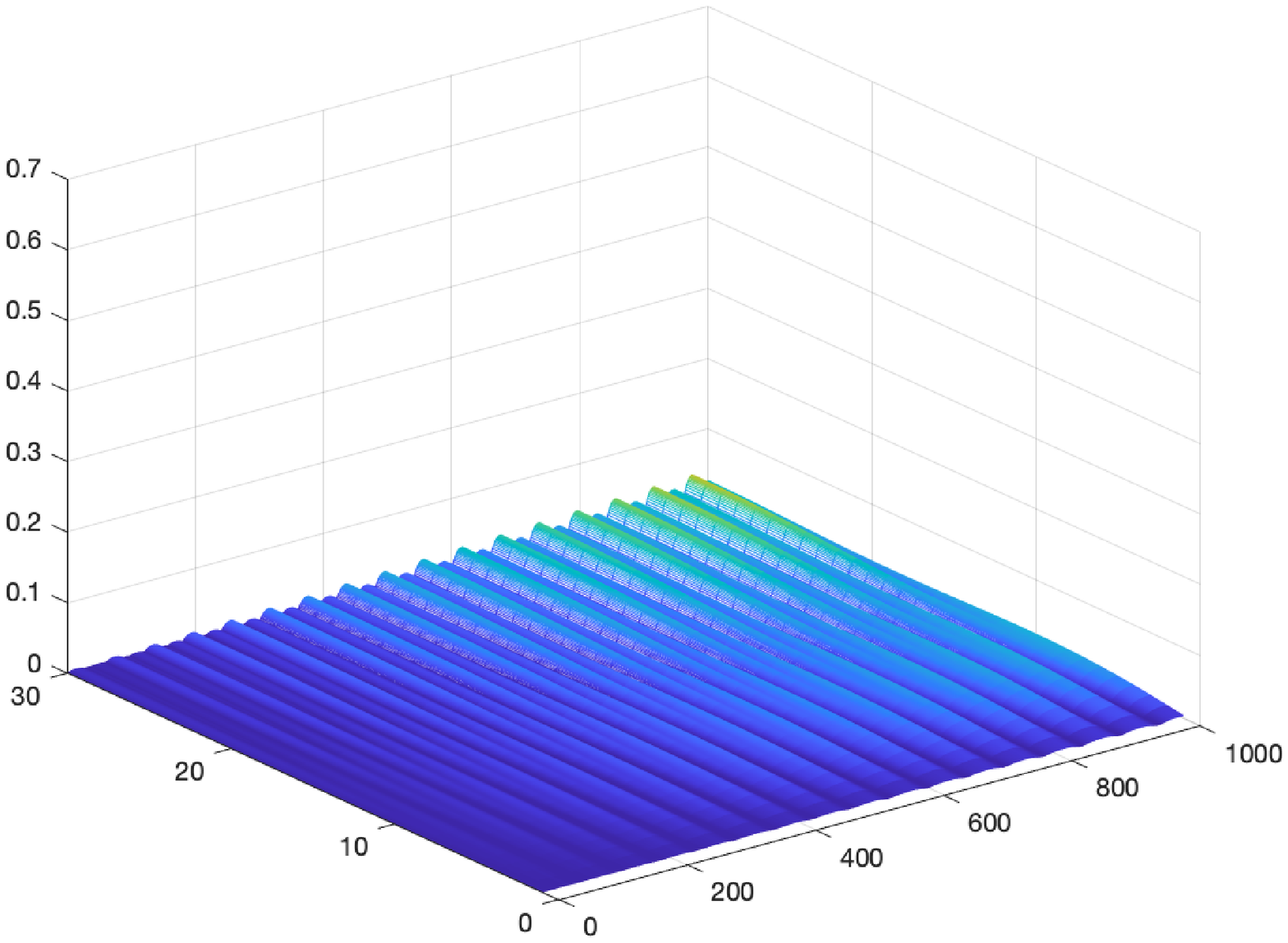}
 \label{fig9}
\end{subfigure}
\caption{All three controlled trajectories from top to bottom respectively, unfolded along the first mode. We plot $t = 0$ (left), $t = 0.5$ (middle) and $t = 1$ (right).}
\label{fig:images}
\end{figure}

\section{Conclusions}
In this paper we have introduced a new algorithm for approximating optimal feedback controls related to optimal control problems driven by evolutionary partial differential equations. The new algorithm is based on a tree structure to avoid the construction of a grid in the solution of the HJB equations, and exploits the compact representation of the dynamical systems based on tensor notations via a higher-order model reduction approach. We have shown how the algorithm can be constructed for general nonlinear control problems and given some crucial hints on its implementation. Furthermore, we have studied the existing pruning techniques for reducing the cardinality of the constructed tree, and introduced a new \textit{statistical pruning} technique for a further reduction in the cardinality of the tree. To guarantee the convergence of the method, we derived an error estimate depending on the time step and on the accuracy of the HO-POD-DEIM basis. Finally, numerical tests on a number of challenging benchmark problem have been discussed, indicating the power of the method, with large savings both in terms of computational time and, especially, memory. We believe that these promising numerical results brings us one step closer to the application of DP in challenging, industrial settings.

To this end, we plan to, in the near future, explore more challenging industrial problems where the combination of the compact tensor representation of the problem and the tree structure algorithm can give a competitive advantage to DP for feedback control problems over possible competitors.

\section*{Acknowledgements}
A large part of this project was started together with Prof. Maurizio Falcone, shortly before he passed away. We greatly acknowledge his contributions, inspirations and leadership in this project. May his memory live a long time in the mathematical community.

\vspace{3mm}

\noindent\textbf{Funding.} The second author is a member of INDAM GNCS (Gruppo Nazionale di Calcolo Scientifico).
This research has been partially supported by the INdAM-GNCS project ”Finanziamento Giovani Ricercatori 2020-2021”.

\vspace{3mm}

\noindent\textbf{Data access.} Matlab codes implementing the numerical examples are available at \url{https://github.com/saluzzi/Multilinear_HJB_POD}.

 \section*{Declarations}
\textbf{Conflict of interest.} The authors declare that they have no conflict of interest.

\begin{appendices}

\section{Proof of Proposition \ref{prop_estimate}}
\label{appendixA}

\begin{proof}
For the sake of simplicity we are going to define $\hat{{L}}=V_Y^\top L V_Y$, $f_{k}={\bm  f}(y^{k},t^{k},u^{k})$, $\hat{f}_{k}=V_{Y}^\top \mathbb{P}  {\bm  f}(V_{Y} \hat{y}^{k},u^{k},t^{k})$ and $f_{k,V}=V_{Y}^\top \mathbb{P}  {\bm  f}(V_{Y} V^\top_{Y} y^{k},u^{k},t^{k})$, where we are considering the same control sequence $\{u_k\}_{k=0}^{N_t-1}$.

We consider the error at time $t_j$ between the full model and the lifted reduced model as
$$
E_j=y^j-V_Y\hat{y}^j
$$
and we rewrite it as a sum of two quantities
$$
E_j=\rho_j+\theta_j
$$
where
$$
\rho_j=y^j-V_YV_Y^\top y^j, \quad\theta_j=V_YV_Y^\top y^j-V_Y\hat{y}^j.
$$

 Multiplying  \eqref{semi_implicit_full} by $V_Y^\top$ and adding and subtracting $\hat{L} V_Y^\top  y^j +   f_{j-1,V}$ we get

$$
V_Y^\top\frac{y^j-y^{j-1}}{\Delta t} = \hat{L} V_Y^\top  y^j +   f_{j-1,V} + \hat{R}_j,
$$
with
$$
\hat{R}_j=V_Y^\top L y^{j} + V_Y^\top f_{j-1}-\hat{L} V_Y^\top  y^j-f_{j-1,V}.
$$
Defining $\hat{\theta}_j=V_Y^\top \theta_j$, we obtain
$$
\frac{\hat{\theta}_j-\hat{\theta}_{j-1}}{\Delta t} =  V_Y^\top\frac{y^j-y^{j-1}}{\Delta t} -  \frac{\hat{y}^j-\hat{y}^{j-1}}{\Delta t} =  \hat{L} V_Y^\top  y^j +   f_{j-1,V} + \hat{R}_j - \hat{L} \hat{y}^{j}-\hat{f}_{j-1}.
$$
Since $\left< \hat{\theta}_j, \hat{\theta}_{j-1} \right> \le \Vert \hat{\theta}_j \Vert \Vert \hat{\theta}_{j-1} \Vert$, we get
$$
\frac{\Vert \hat{\theta}_j \Vert - \Vert \hat{\theta}_{j-1} \Vert}{\Delta t} \le \frac{1}{\Vert \hat{\theta}_j \Vert} \left< \hat{\theta}_j, \frac{\hat{\theta}_j-\hat{\theta}_{j-1}}{\Delta t} \right>
$$
$$
=\frac{1}{\Vert \hat{\theta}_j \Vert} \left( \left< \hat{\theta}_j,\hat{L}\left( V_Y^\top  y^j-  \hat{y}^{j}\right) \right> + \left< \hat{\theta}_j,f_{j-1,V}-\hat{f}_{j-1}+ \hat{R}_j \right> \right)
$$
$$
\le \mu(\hat{L}) \Vert \hat{\theta}_j \Vert +\gamma \Vert \hat{\theta}_{j-1} \Vert + \Vert \hat{R}_j \Vert,
$$
where $\gamma =  L_{\bf f} \Vert V_{Y}^\top \mathbb{P} \Vert $ and we used the definition of logarithmic norm  \eqref{log_norm} and the Lipschitz-continuity of the function ${\bf f}$.
Defining $\zeta=\frac{1}{1-\Delta t \mu(\hat{L})}$ and $\eta=1+\Delta t \gamma$ and by the fact that $\Vert \theta_j \Vert = \Vert \hat{\theta}_j \Vert $, it follows
$$
\Vert \theta_j \Vert   \le\zeta \eta \Vert \theta_{j-1} \Vert + \Delta t \, \zeta \Vert \hat{R}_j \Vert 
\le (\zeta \eta)^j \Vert \theta_0 \Vert + \Delta t \sum_{k=1}^j \zeta^{k} \eta ^{k-1} \Vert \hat{R}_{j-k+1} \Vert 
$$
$$
\le \Delta t \zeta \left( \sum_{k=0}^{j-1} (\zeta \eta)^{2k} \sum_{k=1}^j \Vert \hat{R}_{k} \Vert^2 \right)^{1/2},
$$
where we note that $\theta_0 = 0$ and $\zeta$ is positive due to the assumption on the time step $\Delta t$.
Let us define $q =\sum_{k=0}^{N_t-1} (\zeta \eta)^{2k} $.
Recalling the definition of $\hat{R}_k$
$$
\hat{R}_k=V_Y^\top L\left( y^{k} - V_Y V_Y^\top y^k \right) + V_Y^\top \left(f_{k-1}- \mathbb{P} {\bf f}(V_Y V^\top_Y y^{k-1},u^{k-1},t^{k-1})\right),
$$
we note that
$$
\Vert V_Y^\top \left( f_{k-1}-  \mathbb{P} {\bf f}(V_Y V^\top_Y y^{k-1},u^{k-1},t^{k-1}) \right)\Vert
$$
$$
= \Vert V_Y^\top \left( f_{k-1}-\mathbb{P} f_{k-1}+\mathbb{P} f_{k-1} -\mathbb{P} {\bf f}(V_Y V^\top_Y y^{k-1},u^{k-1},t^{k-1})\right) \Vert
$$
$$
\le {\bf c} \Vert V_Y^\top\Vert \Vert f_{k-1} - V_F V_F^\top f_{k-1} \Vert + \gamma  \Vert \rho_{k-1} \Vert
$$
where we applied Proposition 1 from \cite{kirsten.22} with ${\bf c} = \prod_{m=1}^d \Vert ({\bf P_m}^\top {\bf \Phi_m })^{-1} \Vert$ and the Lipschitz-continuity of the function ${\bf f}$ .
Therefore, we obtain the following upper bound for the term $\hat{R}_{k}$
$$
\Vert \hat{R}_{k} \Vert \le \alpha \Vert \rho_k \Vert  + \beta \Vert w_{k-1} \Vert
$$
where $\alpha=\Vert V_Y^\top L\Vert + \gamma $, $\beta={\bf c}\Vert V_Y^\top \Vert $ and $w_{k-1} = f_{k-1} - V_F V_F^\top f_{k-1}$.

From these results we can get the following estimate for the generic term $\theta_j$

$$
\Vert \theta_j \Vert^2 \le (\Delta t \zeta)^2 q \sum_{k=1}^{j} \Vert \hat{R}_k \Vert^2 \le  2(\Delta t \zeta)^2 q \sum_{k=1}^{j} (\alpha^2\Vert \rho_k \Vert^2  + \beta^2 \Vert w_{k-1} \Vert^2)
$$
 and finally
$$
\sum_{j=0}^{N_t} \Vert E_j \Vert^2 = \sum_{j=0}^{N_t} \Vert \rho_j \Vert^2+  \sum_{j=1}^{N_t} \Vert \theta_j \Vert^2 \le C(T) \left( \mathcal{E}_y + \mathcal{E}_f \right)
$$
where
$$
C(T)=\max\{1+2q\zeta^2 T \Delta t \, \alpha^2, 2q \zeta^2 T \Delta t  \beta^2 \}.
$$

\end{proof}

\begin{rmk}
    Supposing that $\gamma \le-\mu(\hat{L}) $, 
    then  $\zeta \eta<1$ and we obtain the following upper bound for the quantity $q$
$$
q =\sum_{k=0}^{N_t-1} (\zeta \eta)^{2k} \le \frac{1}{1-(\zeta \eta)^{2N_t}} .
$$
\end{rmk}
\begin{rmk}
The constant $C(T)$ depends on the coefficient 
${\bf c} = \prod_{m=1}^d \Vert ({\bf P_m}^\top {\bf \Phi_m })^{-1} \Vert,$
which is minimized applying the {\tt q-deim} procedure, we refer to \cite{gugercin2018} for more details.
\end{rmk}

\end{appendices}

\bibliographystyle{spmpsci} 
\bibliography{deimbib}

\end{document}